\RequirePackage{fix-cm}

\documentclass[smallextended]{svjour3}       % onecolumn (second format)
\smartqed  % flush right qed marks, e.g. at end of proof
\usepackage{graphicx}

\usepackage{epsfig} % for postscript graphics files
\usepackage{mathptmx} % assumes new font selection scheme installed
\usepackage{times} % assumes new font selection scheme installed
\usepackage{amsmath} % assumes amsmath package installed
\usepackage{amssymb}

\usepackage{tikz}
\usetikzlibrary{arrows}
\usetikzlibrary{graphs,decorations.pathmorphing,decorations.markings}
\usetikzlibrary{calc}

\pgfmathsetmacro\weight{1/2}
\pgfmathsetmacro\third{1/3}
\pgfmathsetmacro\twothirds{2/3}

\tikzset{degil/.style={
            decoration={markings,
            mark= at position 0.5 with {
                  \node[transform shape] (tempnode) {$/$};
                  }
              },
              postaction={decorate}
}
}

\usepackage{pdfsync}

\usepackage{enumerate}
\usepackage{marginnote}  % For marginnotes

\newcounter{syscounter}
\newenvironment{sysnum}{\begin{list}{($\Sigma{\arabic{syscounter}}$)}%
{\settowidth{\labelwidth}{($\Sigma4$)}
\settowidth{\leftmargin}{($\Sigma4$)~}%
\usecounter{syscounter}}}
{\end{list}}

\newcommand \A   {0}

\newcommand \N   {\mathbb{N}}
\newcommand \R   {\mathbb{R}}

\newcommand \Z   {\mathbb{Z}}

\newcommand \K   {\mathcal{K}}
\newcommand \Kinf{\mathcal{K_\infty}}

\newcommand \KL  {\mathcal{KL}}
\newcommand \LL  {\mathcal{L}}

\newcommand \Dc   {\mathcal{D}}

\newcommand \qrq   {\quad\Rightarrow\quad}
\newcommand \srs   {\,\Rightarrow\,}

\newcommand \Iff   {\Leftrightarrow}

\newcommand \eps {\varepsilon}

\newcommand \Sigmafp {0}

 \journalname{Mathematics of Control, Signals, and Systems}

\begin{document}

%\title{Integral uniform global asymptotic stability and non-coercive Lyapunov functions
\title{Existence of non-coercive Lyapunov functions is equivalent to integral uniform global asymptotic stability 
%\thanks{
%}
}

\titlerunning{Integral UGAS and non-coercive Lyapunov functions
}        % if too long for running head

%\author{Sergey Dashkovskiy         \and
        %Andrii Mironchenko 
%}

\author{Andrii Mironchenko  \and Fabian Wirth       }


%\institute{Sergey Dashkovskiy and Andrii Mironchenko (corresponding author) \at
              %Faculty of Mathematics and Computer Science, University of Bremen, Bibliothekstra\ss e 1, 28334 Bremen, Germany \\
              %Tel.: +49-421-21863745\\
              %Fax: +49-421-2189863745\\
              %\email{ \{dsn,andmir\}@math.uni-bremen.de}           %  \\
%%             \emph{Present address:} of F. Author  %  if needed
%}

\institute{Andrii Mironchenko (corresponding author)  \at
							Faculty of Computer Science and Mathematics, University of Passau, Innstra\ss e 33, 94032 Passau, Germany\\
%              Tel.: +49-421-21863745\\
%              Fax: +49-421-2189863745\\
              \email{andrii.mironchenko@uni-passau.de}            %  \\
           \and
           Fabian Wirth \at
							Faculty of Computer Science and Mathematics, University of Passau, Innstra\ss e 33, 94032 Passau, Germany\\
              Tel.: +49-421-21863825\\
              Fax: +49-421-2189863825\\
              \email{fabian.(lastname)@uni-passau.de}           %  \\
}

\date{Received: date / Accepted: date}

\maketitle

\begin{abstract}
In this paper a class of abstract dynamical systems is considered which
encompasses a wide range of nonlinear finite- and infinite-dimensional systems.
We show that the existence of a non-coercive Lyapunov function without any further requirements on the flow of the forward complete system ensures an integral version of uniform global asymptotic stability. We prove that also the converse statement holds without any further requirements on regularity of the system.

Furthermore, we give a characterization of uniform global asymptotic stability in terms of
the integral stability properties and analyze which stability properties
can be ensured by the existence of a non-coercive Lyapunov function, provided either the flow has a kind of uniform continuity near the equilibrium or the system is robustly forward complete.

%Input-to-state stability, Lyapunov methods, nonlinear control systems, reaction-diffusion systems, linearization, semigroup theory.
\end{abstract}

%\paragraph{Paragraph headings} Use paragraph headings as needed.

%%%%%%%%%%%%%%%%%%%%%%%%%%%%%%%%%%%%%%%%%%%%%%%%%%%%%%%%%%%%%%%%%%%%%%%%%%%%%%%%

%{\color{red}
%Graph in the conclusions should be redesigned so that it fits MCSS layout.
%
%Question to Fabian:
%
%Do we want to put this figure before the proofs of the main results, e.g. after the definition of all the notions?
%
%}

\section{Introduction}
\label{sec:introduction}

The theory of Lyapunov functions is one of the cornerstones in the
dynamical and control systems theory.  Numerous applications of Lyapunov
theory include characterization of stability properties of fixed points
and more complex attractors \cite{Yos66,Hah67,LSW96,Kel15}, conditions for
forward completeness of trajectories \cite{AnS99}, criteria for the
existence of a bounded absorbing ball \cite[Theorem 2.1.2]{Chu15} etc.  Some of these uses extend from finite-dimensional
applications to the infinite-dimensional case, while others rely on distinct
finite-dimensional arguments.

On the other hand numerous converse results have been obtained which prove
the existence of certain types of Lyapunov functions characterizing
different stability notions. Indeed, before starting to look for a Lyapunov function it is highly desirable to
know in advance that such a Lyapunov function for a given class of systems
exists.  The first results guaranteeing existence of Lyapunov functions
for asymptotically stable systems appeared in the works of Kurzweil
\cite{Kur56} and Massera \cite{Mas56}.  These have been
generalized in different directions, see \cite{Kel15,MiW17a} for references.

The standard definition of a Lyapunov
function $V$, found in many textbooks on finite-dimensional dynamical
systems, is that it should be a continuous (or more regular) positive
definite and proper function, i.e. a function for which there exist
$\Kinf$\footnote{An increasing, unbounded, continuous, positive definite
  function from $\R_+$ to itself that maps $0$ to $0$.} functions $\psi_1,\psi_2,\alpha$ such that
\begin{equation}
    \label{eq:2}
    \psi_1(\|x\|) \leq V(x) \leq \psi_2(\|x\|) \quad \forall x \in X,
\end{equation}
and such that 
\begin{equation}
    \label{eq:3}
    \dot V(x) < - \alpha(\|x\|) \quad \forall x \in X, 
\end{equation}
where $\dot V(x)$ is some sort of generalized derivative of $V$ along the
trajectories of the system.

If $V$ is as above with the exception that instead of \eqref{eq:2}, $V$
satisfies the weaker property
\begin{equation}
    \label{eq:2b}
    0 < V(x) \leq \psi_2(\|x\|)\,,\quad x\neq 0,
\end{equation}
then $V$ is called a non-coercive Lyapunov function.

Noncoercive Lyapunov functions are frequently used in the linear infinite-dimensional systems theory. 
There are at least two reasons for this. On the one hand, using the
generalized Datko lemma \cite{Dat70,Lit89} one can show that the existence of noncoercive Lyapunov functions already proves exponential stability of a linear system (and thus it is not necessary to look for coercive Lyapunov functions). On the other hand, noncoercive Lyapunov functions are in a certain sense even more natural than coercive ones. For example, a classic type of Lyapunov functions for linear exponentially stable systems over Hilbert spaces are quadratic Lyapunov functions constructed by solving the operator Lyapunov equation \cite[Theorem 5.1.3]{CuZ95}. However, solutions of this equation are not coercive in general, and hence the corresponding Lyapunov functions are not coercive as well. 

In spite of these advantages, the usage of non-coercive Lyapunov functions was limited to linear infinite-dimensional systems and to nonlinear time-delay systems, for which the efficient method of Lyapunov-Krasovskii functionals is widely used \cite{HaV93,PeJ06} (Lyapunov-Krasovskii functionals have, however, a different type of noncoercivity, see \cite{MiW17d} for a comparison and discussion). 
Recently the situation has changed: in \cite{MiW17a} the authors have
shown that for a broad class of forward complete \textit{nonlinear}
infinite-dimensional systems existence of a non-coercive Lyapunov function
ensures uniform global asymptotic stability (UGAS) of a system, provided
the flow of the system has a certain uniform continuity at the origin and finite-time reachability sets of the system are bounded.
On the other hand, it was demonstrated in \cite{MiW17a} that without these additional assumptions uniform global asymptotic stability cannot be guaranteed, even for systems of ordinary differential equations (ODEs). In particular, the existence of a non-coercive Lyapunov function alone does not ensure forward completeness of the system (in contrast to coercive Lyapunov functions). 
Hence, although non-coercive Lyapunov functions provide more flexibility
for the stability analysis of dynamical systems, further conditions have to be verified separately.
Another result of \cite{MiW17a} is a construction of a Lipschitz continuous non-coercive Lyapunov function by means of an integration of the solution along trajectories. 

In this paper, we continue the investigations initiated in \cite{MiW17a}. In \textit{our first main result} (Theorem~\ref{thm:direct_Non-coercive_Lyapunov_theorem}), we show that forward complete systems possessing non-coercive Lyapunov functions (even if they do not satisfy any further assumptions) enjoy an ``integral version'' of uniform global asymptotic stability (iUGAS), 
which is a weaker notion than UGAS. 
\textit{Our second result}
(Theorem~\ref{thm:converse_Non-coercive_Lyapunov_theorem}) is a converse
non-coercive Lyapunov theorem for the iUGAS property. Since iUGAS is weaker than UGAS, a coercive Lyapunov function does not exist for such systems in general. 
However, we show (without requiring any further regularity of the flow!) that we can construct a non-coercive Lyapunov function for this system. 
The construction is motivated by \cite{MiW17a} and based upon classic converse theorems and
Yoshizawa's method \cite[Theorem 19.3]{Yos66}, \cite[Theorem 4.2.1]{Hen81}.  
A key tool for achieving our main results are the characterizations of the iUGAS property in terms of weaker stability notions,  developed in Theorem~\ref{thm:iUGAS_characterization}, which is a third notable result in this work.
In Figure~\ref{fig:Stability} we provide a graphical overview of the results obtained in this paper, in particular, the relationship between the introduced stability notions.

Relations between integral and ``classic'' stability notions have been studied in a number of papers.
In particular, in \cite{TPL02} uniform global asymptotic stability of finite-dimensional differential inclusions has been characterized via ``integral'' uniform attractivity.
A natural extension of the iUGAS notion to the case of systems with inputs leads to the nonlinear counterparts of $L_2$-stability (which was originally introduced in the context of linear systems in the seminal work \cite{Zam66}, see also \cite{Sch02}).
In \cite{Son98,KeD16} it was shown that these extensions are equivalent to input-to-state stability for the systems of ordinary differential equations with Lipschitz continuous nonlinearities.

\subsection{Notation}

The following notation will be used throughout. By $\R_+$ we
denote the set of nonnegative real numbers. For an arbitrary set $S$ and
$n\in\N$ the $n$-fold Cartesian product is $S^n:=S\times\ldots\times
S$. 
%For normed linear spaces $X,Y$ we denote by $\mathcal{L}(X,Y)$ the space of
%linear bounded operators acting from $X$ to $Y$ and we abbreviate $\mathcal{L}(X):=\mathcal{L}(X,X)$.
The open ball in a normed linear space $X$ endowed with the norm $\|\cdot\|_X$ with radius $r$ and center in $y \in X$ is denoted by $B_r(y):=\{x\in X \ |\  \|x-y\|_X<r\}$ (the space $X$ in which the ball is taken, will always be clear from the context). For short, we denote $B_r:=B_r(0)$. The (norm)-closure of a set $S \subset X$ will be denoted by $\overline{S}$.

For the formulation of
stability properties the following classes of comparison functions are
useful, see \cite{Hah67,Kel14}. The set ${\K}$ is the set of functions $\gamma:\R_+
\to \R_+$ that are continuous, strictly increasing and with
$\gamma(0)=0$; $\K_{\infty}$ is the set of unbounded $\gamma\in\K$; $\KL$
  is the set of continuous $\beta: \R_+^2 \to \R_+$, such that
  $\beta(\cdot,t)\in{\K}$, for all $t \geq 0$ and $ \beta(r,\cdot)$ is
  decreasing to $0$ for all $r >0$.
%\begin{equation*}
% \begin{array}{ll}
% {\K} &:= \left\{\gamma:\R_+ \to \R_+ \left|\ \right. \gamma\mbox{ is continuous and strictly increasing, }\gamma(0)=0\right\},\\
% {\K_{\infty}}&:=\left\{\gamma\in\K\left|\ \gamma\mbox{ is unbounded}\right.\right\},\\
% {\LL}&:=\left\{\gamma:\R_+ \to \R_+ \left|\ \gamma\mbox{ is continuous and decreasing with}\right.
%  \lim\limits_{t\rightarrow\infty}\gamma(t)=0 \right\},\\
% {\KL} &:= \left\{\beta: \R_+^2 \to \R_+ \left|\ \right. \beta(\cdot,t)\in{\K},\ \forall t \geq 0,\  \beta(r,\cdot)\in {\LL},\ \forall r >0\right\}.
% \end{array}
% \end{equation*}

\section{Problem statement}
\label{sec:problem-statement}

We consider abstract axiomatically defined time-invariant
and forward complete systems on the state space $X$ which are subject to a shift-invariant set of
disturbances $\Dc$.
\begin{definition}
\label{Steurungssystem}
Consider the triple $\Sigma=(X,\Dc,\phi)$, consisting of:
\begin{enumerate}[(i)]  
	\item A normed linear space $(X,\|\cdot\|_X)$, called the {state space}, endowed with the norm $\|\cdot\|_X$.
	\item A set of  disturbance values $D$, which is a nonempty subset of a certain normed linear space.
	\item A {space of disturbances} $\Dc \subset \{d:\R_+ \to D\}$
          satisfying the following two axioms.
					
\textit{The axiom of shift invariance} states that for all $d \in \Dc$ and all $\tau\geq0$ the time
shift $d(\cdot + \tau)$ is in $\Dc$.

\textit{The axiom of concatenation} is defined by the requirement that for all $d_1,d_2 \in \Dc$ and for all $t>0$ the concatenation of $d_1$ and $d_2$ at time $t$
\begin{equation}
d(\tau) := 
\begin{cases}
d_1(\tau), & \text{ if } \tau \in [0,t], \\ 
d_2(\tau-t),  & \text{ otherwise},
\end{cases}
\label{eq:Composed_Input}
\end{equation}
belongs to $\Dc$.

	\item A map $\phi:\R_+ \times X \times \Dc \to X$, called the \textit{transition map}.
\end{enumerate}
The triple $\Sigma$ is called a (forward complete) system, if the following properties hold:

\begin{sysnum}
	\item forward completeness: for every $(x,d) \in X \times \Dc$ and
          for all $t \geq 0$ the value $\phi(t,x,d) \in X$ is well-defined.
	\item\label{axiom:Identity} The identity property: for every $(x,d) \in X \times \Dc$
          it holds that $\phi(0, x,d)=x$.
	\item Causality: for every $(t,x,d) \in \R_+ \times X \times
          \Dc$, for every $\tilde{d} \in \Dc$, such that $d(s) =
          \tilde{d}(s)$, $s \in [0,t]$ it holds that $\phi(t,x,d) = \phi(t,x,\tilde{d})$.
	\item \label{axiom:Continuity} Continuity: for each $(x,d) \in X \times \Dc$ the map $t \mapsto \phi(t,x,d)$ is continuous.
		\item \label{axiom:Cocycle} The cocycle property: for all $t,h \geq 0$, for all
                  $x \in X$, $d \in \Dc$ we have
$\phi(h,\phi(t,x,d),d(t+\cdot))=\phi(t+h,x,d)$.
\end{sysnum}
\end{definition}
Here  $\phi(t,x,d)$ denotes the state of the system at the moment $t \in
\R_+$ corresponding to the initial condition $x \in X$ and the disturbance $d \in \Dc$.

%\begin{definition}
%\label{axiom:Lipschitz}
%We say that the flow of $\Sigma=(X,\Dc,\phi)$ is Lipschitz continuous on compact intervals, if 
%for any $\tau>0$ and any $r>0$ there exists $L>0$ so that 
%\begin{equation}
%x,y\in \overline{B_r},\ t \in [0,\tau], d\in\Dc \qrq \|\phi(t,x,d) - \phi(t,y,d) \|_X \leq L \|x-y\|_X.
%\label{eq:Flow_is_Lipschitz}
%\end{equation}	
%\end{definition}

We require a stronger version of forward completeness.
\begin{definition}
\label{Def_RFC}
The system $\Sigma=(X,\Dc,\phi)$ is called robustly forward complete (RFC) 
if for any $C>0$ and any $\tau>0$ it holds that 
\begin{equation*}
 \sup\big\{ \|\phi(t,x,d)\|_X \ |\  \|x\|_X\leq C,\: t \in [0,\tau],\: d\in \Dc \big\} < \infty.  
\end{equation*}
\end{definition}
In other words, a system $\Sigma$ is RFC iff its finite-time reachability sets (emanating from the bounded sets) are bounded.

The condition of robust forward completeness is satisfied by large classes of infinite-dimensional
systems. 

\begin{definition}
\label{def:EqPoint}
We call $0 \in X$ an equilibrium point of the system $\Sigma=(X,\Dc,\phi)$, if 
$\phi(t,0,d) = 0$ for all $t \geq 0$, $d \in \Dc$.
\end{definition}

Note that according to the above definition disturbances cannot
move the system out of the equilibrium position.
\begin{definition}
\label{def:RobustEqPoint}
We call $0 \in X$ a robust equilibrium point (REP) of the system
$\Sigma=(X,\Dc,\phi)$, if it is an equilibrium point such that 
for every $\eps >0$ and for any $h>0$ there exists $\delta = \delta
(\eps,h)>0$, satisfying 
\begin{equation}
\hspace{-7mm} t\in[0,h],\ \|x\|_X \leq \delta, \ d \in \Dc \quad \Rightarrow \quad  \|\phi(t,x,d)\|_X \leq \eps.
\label{eq:RobEqPoint}
\end{equation}
\end{definition}

In this paper we investigate the following stability properties of equilibria of
abstract systems.

\begin{definition}{}
    \label{d:stability_new}
Consider a system $\Sigma=(X,\Dc,\phi)$ with a fixed point $0$. The
equilibrium position $0$ is called
\begin{enumerate}[(i)]
  %\item  Lagrange stable, if there exist $c>0$ and $\sigma \in \Kinf$ so that 
%\[
 %x\in X,\ t \geq 0,\ d\in\Dc \quad\Rightarrow\quad  \|\phi(t,x,d)\|_X \leq \sigma(\|x\|_X) + c.
%\]
  \item uniformly locally stable (ULS), if for every $\eps >0$ there is a $\delta>0$ so that
\begin{eqnarray}
\|x\|_X \leq \delta,\ d\in\Dc,\ t \geq 0 \quad\Rightarrow\quad\|\phi(t,x,d)\|_X \leq \eps.
\label{eq:Uniform_Stability}
\end{eqnarray}
%\item uniformly globally stable, if it is Lagrange stable with $c=0$.
\item uniformly globally asymptotically
stable (UGAS)
if there exists a $\beta \in \KL$ such that 
\begin{equation}
\label{UGAS_wrt_D_estimate}
x \in X,\ d\in \Dc,\ t \geq 0  \srs \|\phi(t,x,d)\|_X \leq \beta(\|x\|_X,t).
\end{equation}
\item uniformly (locally) asymptotically stable (UAS)
if there exist a $\beta \in \KL$ and an $r>0$ such that 
\begin{equation*}
%\label{UAS_wrt_D_estimate}
\|x\|_X\leq r,\ d\in \Dc,\ t \geq 0  \srs \|\phi(t,x,d)\|_X \leq \beta(\|x\|_X,t).
\end{equation*}

%\item globally weakly attractive (GWA), if
%\begin{eqnarray}
%x\in X,\ d \in\Dc \quad\Rightarrow\quad \inf_{t \geq 0} \|\phi(t,x,d)\|_X = 0.
%\label{eq:LIM_inf_form}
%\end{eqnarray}

\item \label{def:UniformGlobalWeakAttractivity} uniformly globally weakly attractive (UGWA), if for every $\eps>0$ and
for every $r>0$ there exists a $\tau = \tau(\eps,r)$ such that for all $\|x\|_X\leq r,\ d\in\Dc$
\begin{equation*}
 \exists t = t(x,d,\eps) \leq \tau:\ \|\phi(t,x,d)\|_X \leq \eps.
%\label{eq:Uniform_weak2}
\end{equation*}
\item \label{def:UniformGlobalAttractivity} uniformly globally attractive (UGATT), if for any $r,\eps >0$ there exists $\tau=\tau(r,\eps)$ so that
\begin{equation*}
\|x\|_X \leq r,\ d\in\Dc,\ t \geq \tau(r,\eps) \quad \Rightarrow \quad \|\phi(t,x,d)\|_X \leq \eps.
%\label{eq:UAG_with_zero_gain}
\end{equation*}
\end{enumerate}
\end{definition}
It is clear, that UGAS of $0$ implies UGATT of $0$, which in turn implies UGWA of $0$.

As we will see, in the study of non-coercive Lyapunov functions one
arrives very naturally at ``integral'' versions of the notions stated above:
\begin{definition}
\label{def:iREP}
We call $0 \in X$ an integrally robust equilibrium point (iREP) of the system
$\Sigma=(X,\Dc,\phi)$, if it is an equilibrium point and there is $\alpha\in\K$ such that 
for every $\eps >0$ and for any $h>0$ there exists $\delta = \delta
(\eps,h)>0$, satisfying 
\begin{equation}
\hspace{-7mm} \|x\|_X \leq \delta, \ d \in \Dc \quad \Rightarrow \quad \int_0^h \alpha(\|\phi(s,x,d)\|_X)ds \leq \eps.
\label{eq:iREP}
\end{equation}
\end{definition}

\begin{definition}
\label{Def_iRFC}
For a given $\alpha\in\K$, system $\Sigma=(X,\Dc,\phi)$ is called $\alpha$-integrally robustly forward complete ($\alpha$-iRFC),
if for any $C>0$ and any $\tau>0$ it holds that 
\[
\sup_{x\in\overline{B_C},\ d\in\Dc }    \int_0^\tau \alpha\big(\|\phi(t,x,d)\|_X\big) dt < \infty.
\]
\end{definition}

\begin{remark}
\label{rem:RFC_and_iRFC} 
Note that every forward-complete system is automatically $\alpha$-iRFC for
any bounded $\alpha\in\K$, since
$\int_0^\tau \alpha\big(\|\phi(t,x,d)\|_X\big) dt <\tau \sup_s \alpha(s)$.
On the other hand, if $\Sigma$ is RFC, then $\Sigma$ is also $\alpha$-iRFC for any $\alpha\in\K$.
\end{remark}

\begin{definition}
\label{def:Integral_stability_notions} 
Consider a forward complete system $\Sigma=(X,\Dc,\phi)$ with a fixed
point at $0$. 
The equilibrium position $0$ is called
\begin{enumerate}[(i)]
\item integrally uniformly locally stable (iULS) provided there are $\alpha\in\K$, $\psi\in\Kinf$ and $r>0$ so that 
\begin{equation}
\|x\|_X\leq r,\ d\in\Dc \srs \int_0^\infty \alpha(\|\phi(s,x,d)\|_X)ds  \leq \psi(\|x\|_X).
\label{eq:Integral_boundedness}
\end{equation}
\item integrally uniformly globally stable (iUGS) provided there are $\alpha\in\K$, $\psi\in\Kinf$ so that \eqref{eq:Integral_boundedness} is valid for $r:=\infty$.
\item integrally uniformly globally attractive (iUGATT)  provided there is $\alpha\in\K$ so that
\begin{eqnarray}
\forall \ r>0 \,:\, \lim_{t\to\infty} \sup_{x\in \overline{B_r},\;d\in\Dc} \int_t^\infty \alpha(\|\phi(s,x,d)\|_X)ds  = 0.
\label{eq:integral_UGATT}
\end{eqnarray}
\item integrally uniformly globally asymptotically stable (iUGAS)
  provided there are $\alpha\in\K$ and $\beta\in\KL$ so that for all $x\in
  X,\ d\in\Dc,\ t\geq 0$ we have
\begin{equation}
%x\in X,\ d\in\Dc,\ t\geq 0 \srs 
\int_t^\infty \alpha(\|\phi(s,x,d)\|_X)ds  \leq \beta(\|x\|_X,t).
\label{eq:integral_UGAS}
\end{equation}
\end{enumerate}
\end{definition}

%We will need one more notion:
%\begin{definition}
%\label{def:} 
%We say that $\Sigma$ has integrally bounded reachability sets if there is $\alpha\in\K$ so that
%\begin{eqnarray}
%R>0,\ \tau >0 \srs \sup_{x\in\overline{B_R},\; d\in\dc} \int_0^\tau 
%\label{eq:}
%\end{eqnarray}
%
%\end{definition}

\smallskip{}

Properties \eqref{eq:Integral_boundedness} and \eqref{eq:integral_UGATT} resemble a kind of uniform attractivity.
This similarity becomes even more apparent if we rewrite the definition of UGATT in an equivalent form:
\begin{lemma}
\label{lem:UGATT_restatement} 
Let $\Sigma=(X,\Dc,\phi)$ be a forward complete system with fixed point $\Sigmafp$. Then $\Sigmafp$ is UGATT iff there is $\alpha\in\K$ so that 
\begin{eqnarray}
\label{eq:UGATT_restatement}
\lim_{t\to\infty }\sup_{x\in\overline{B_r},\; d\in\Dc}  \alpha\Big(\big\|\phi(\cdot + t,x,d)\big\|_{C(X)}\Big) =0\quad \forall r>0,
\end{eqnarray}
where $\|\phi(\cdot + t,x,d)\|_{C(X)}$ is the $\sup$-norm of the ``tail'' of the trajectory $\phi$ after the time $t$.
\end{lemma}

\begin{proof}
If $\Sigmafp$ is UGATT, then for any $\alpha\in\Kinf$ and any $r,\eps >0$ there exists $\tau=\tau(r,\eps)$ so that
\begin{eqnarray*}
\|x\|_X \leq r,\ d\in\Dc,\ t \geq \tau(r,\eps) \quad \Rightarrow \quad \|\phi(t,x,d)\|_X \leq \alpha^{-1}(\eps).
\end{eqnarray*}
%Hence, 
%\begin{eqnarray*}
%\|x\|_X \leq r,\ d\in\Dc,\ t \geq \tau(r,\eps) \quad \Rightarrow \quad \alpha\big(\|\phi(t,x,d)\|_X\big%) \leq \eps
%\end{eqnarray*}
Equivalently, the left hand side implies $\alpha\big(\|\phi(t,x,d)\|_X\big) \leq \eps$
and taking the limit $\eps\to+0$ we arrive at \eqref{eq:UGATT_restatement}. The proof of the converse implication is analogous.
\end{proof}

\begin{remark}
\label{rem:with_PD_alpha_we_do_not_get_UGATT} 
Note that merely choosing a positive definite $\alpha$ in
\eqref{eq:UGATT_restatement} (i.e. $\alpha\in C(\R_+,\R_+)$: $\alpha(0)=0$ and $\alpha(r)>0$ for $r>0$) we do not arrive at any kind of attractivity, since 
the trajectory may grow to infinity, and $\alpha(\|\phi(t,x,d)\|_X)$ may converge to zero at the same time.
E.g. consider $\dot{x}(t)=x(t)$, $x(t)\in\R$, $\alpha(r):=\frac{r}{r^2+1}$. 
\end{remark}

Analogously, it is possible to restate the UGS property. 
In Theorem~\ref{thm:UGAS_characterization} we will show that UGAS implies iUGAS.

Finally, it is easy to see that 
\begin{lemma}
\label{lem:REP_implies_iREP} 
Let $\Sigma=(X,\Dc,\phi)$ be a system. If $\Sigmafp$ is a REP, then
$\Sigmafp$ is an iREP with arbitrary $\alpha\in\Kinf$.
\end{lemma}

\begin{proof}
Fix $\alpha\in\Kinf$. Since $\Sigmafp$ is a REP of
$\Sigma=(X,\Dc,\phi)$, for every $\eps >0$, $h>0$ there is $\delta =
\delta(\eps,h)>0$ such that 
\begin{equation*}
\hspace{-7mm} \|x\|_X \leq \delta, \ d \in \Dc \quad \Rightarrow \quad \sup_{t\in[0,h]} \|\phi(s,x,d)\|_X \leq \alpha^{-1}\big(\frac{\varepsilon}{h}\big).
\end{equation*}
Hence it holds that
\begin{equation*}
\hspace{-7mm} \|x\|_X \leq \delta, \ d \in \Dc \quad \Rightarrow \quad \int_0^h \alpha\big(\|\phi(s,x,d)\|_X\big) ds \leq \varepsilon,
\end{equation*}
which shows $\Sigmafp$ is an iREP with the above $\alpha$.
\end{proof}

We now introduce Lyapunov functions which will help in characterizing the
UGAS and iUGAS concepts. To this end we first recall the notion of Dini derivative.
For % a continuous function
$h: \R \to \R$ 
the right-hand lower Dini derivative $D_+$ and the right-hand upper Dini
derivative $D^+$ at a point $t\in \R$ are
defined by, see \cite{Sza65},
\begin{equation}
    \label{eq:4}
    \begin{aligned}
       D_+h(t)&:=\mathop{\underline{\lim}} \limits_{\tau \rightarrow +0}
       {\frac{1}{\tau}\big(h(t+\tau)-h(t)\big) },
\\ 
		D^+h(t)&:=\mathop{\overline{\lim}} \limits_{\tau \rightarrow +0} {\frac{1}{\tau}\big(h(t+\tau)-h(t)\big) }.     
    \end{aligned}
\end{equation}

Consider a system $\Sigma= (X,\Dc,\phi)$ and let $V:X \to\R$ be a map.
%be continuous
%along trajectories of $\Sigma$, i.e. the map $t\mapsto V(\phi(t,x,d))$ is
%continuous for all $(x,d)$.
Given $x\in X,d\in \mathcal{D}$, we consider the (right-hand lower) Dini
derivative of the function $t \mapsto V(\phi(t,x,d))$ at $t=0$ denoted by:
\begin{equation}
\label{UGAS_wrt_D_LyapAbleitung}
\dot{V}_d(x):=\mathop{\underline{\lim}} \limits_{t \rightarrow +0} {\frac{1}{t}\Big(V\big(\phi(t,x,d)\big)-V(x)\Big) }.
\end{equation}
We call this the Dini derivative of $V$ along the trajectories of
$\Sigma$. We stress that at this point no continuity assumption has been
placed on $V$.

Having introduced the main stability properties, we introduce now a predominant tool for their study, which is a Lyapunov function.
\begin{definition}
\label{def:UGAS_LF_With_Disturbances}
Consider a system $\Sigma= (X,\Dc,\phi)$ and a function $V:X \to \R_+$, satisfying for each $y\in X$, each $s> 0$ and each $d\in\Dc$ the inequalities
%\begin{eqnarray}
%\label{eq:V_estimate_along_trajectory_final}
%\mathop{\underline{\lim}} \limits_{h \rightarrow +0} V\big(\phi(s-h,y,d)\big) &\geq& V\big(\phi(s,y,d)\big)
%\\ &\geq& 
%\mathop{\underline{\lim}} \limits_{h \rightarrow +0}
%V\big(\phi(s+h,y,d)\big).
%\nonumber
%\end{eqnarray}
\begin{eqnarray}
\label{eq:V_estimate_along_trajectory_final}
\mathop{\underline{\lim}} \limits_{h \rightarrow +0} V\big(\phi(s-h,y,d)\big) \geq V\big(\phi(s,y,d)\big)
\geq
\mathop{\underline{\lim}} \limits_{h \rightarrow +0}
V\big(\phi(s+h,y,d)\big).
\nonumber
\end{eqnarray}
Assume also that the right inequality in \eqref{eq:V_estimate_along_trajectory_final} is satisfied for $s:=0$ as well.
The map $V$ is called:
\begin{enumerate}[(i)] 
	\item a \textit{non-coercive Lyapunov function} for the system
          $\Sigma=(X,\Dc,\phi)$,  if $V(0)=0$ and if there exist $\psi_2 \in \Kinf$ and $\alpha \in \K$ such that 
\begin{equation}
    \label{eq:1}
    0 < V(x) \leq \psi_2(\|x\|_X) \quad \forall x \in X \backslash\{0\}.
\end{equation}
holds and the Dini derivative of $V$  along the trajectories of $\Sigma$ satisfies 
\begin{equation}
\label{DissipationIneq_UGAS_With_Disturbances}
\dot{V}_d(x) \leq -\alpha(\|x\|_X)
\end{equation}
for all $x \in X$ and all $d \in \Dc$. 

% {\color{blue}
%    \item \textit{semi-coercive Lyapunov function} if in addition there is $\psi_1\in\K$ so that
%  \begin{equation}
%  \label{LyapFunk_1Eig_UGAS}
%  \psi_1(\|x\|_X) \leq V(x) \leq \psi_2(\|x\|_X) \quad \forall x \in X.
%  \end{equation}
% }  
	 \item a \textit{(coercive) Lyapunov function} if in addition
           there is $\psi_1\in\Kinf$ satisfying \\
					$\psi_1(\|x\|_X) \leq V(x)$ for all $x\in X$. 
%\eqref{LyapFunk_1Eig_UGAS}
\end{enumerate}
 
\end{definition}
The inequalities \eqref{eq:V_estimate_along_trajectory_final}
 say that if a Lyapunov function is not continuous along a trajectory at some point, then its value jumps down at this point.

%{\color{red}
%Is semi-coercivity needed for the paper?
%FW: not for the MTNS, but the connection to local stability is still interesting.
%}

%If we want to emphasize that \eqref{LyapFunk_1Eig_UGAS} holds we will also
%speak of a coercive Lyapunov function. 

The following result is known:
\begin{proposition}
\label{Direct_LT_0-UGAS_maxType}
Let $\Sigma=(X,\Dc,\phi)$ be a system. Then:
\begin{itemize}
	\item[(i)] If there exists a coercive continuous Lyapunov function for $\Sigma$, then $\Sigmafp$ is UGAS.
	\item[(ii)] If there exists a non-coercive continuous Lyapunov function for $\Sigma$, and if $\Sigma$ is RFC and $\Sigmafp$ is a robust equilibrium, then $\Sigmafp$ is UGAS.
\end{itemize}
\end{proposition}
Proposition~\ref{Direct_LT_0-UGAS_maxType}\,(i) is a classic result, and item (ii) has been shown in
\cite{MiW17a}, where the concept of a non-coercive Lyapunov function for nonlinear systems has been introduced and analyzed.
% \textit{There is an apparent distinction in the results (i) and (ii).}
% Already the existence of a coercive Lyapunov function alone (without any
% further assumptions) suffices for UGAS of $\Sigmafp$. On the other hand,
% in item (ii) existence of a non-coercive Lyapunov function again implies
% UGAS, but the result requires that REP and RFC hold. In case that either REP or RFC do not hold, non-coercive Lyapunov functions do not imply UGAS anymore, as demonstrated by several examples in \cite{MiW17a}.
\textit{There is an apparent distinction in the results (i) and (ii)}, in that
in item (ii) the existence of a non-coercive Lyapunov function implies
UGAS, provided that REP and RFC hold. In case that either REP or RFC do not hold, non-coercive Lyapunov functions do not imply UGAS, as demonstrated by examples in \cite{MiW17a}.

This difference in the formulations of items (i) and (ii) of Proposition~\ref{Direct_LT_0-UGAS_maxType} motivates the first question:
\begin{center}
\textit{What are the stability properties, which can be inferred from the existence of a non-coercive Lyapunov function, without requiring any further assumptions on $\Sigma$?}
\end{center}
On the other hand, it is well-known, that UGAS implies existence of a coercive Lyapunov function, at least under certain regularity assumptions on the flow of $\Sigma$. This leads to the second problem which we analyze in this paper:
\begin{center}
\textit{What property, which is weaker than UGAS, implies existence of a non-coercive Lyapunov function (and at the same time does not imply the existence of a coercive Lyapunov function)?}
\end{center}
In Section~\ref{sec:Non-coercive_LFs} we resolve both these questions by
showing that \emph{existence of a non-coercive Lyapunov function is equivalent
to the iUGAS property}. Moreover, in
Section~\ref{sec:iUGAS_criterion_and_Noncoercivity} we show several useful
criteria for iUGAS and iUGATT, we give ``atomic decompositions'' of the UGAS
property in Section~\ref{sec:UGAS_characterizations}. Furthermore, in
Section~\ref{sec:Non-coercive_LFs} we analyze which stability properties
can be ensured by the existence of a non-coercive Lyapunov function, if
it is only assumed that either $\Sigmafp$ is a REP or that the RFC property of $\Sigma$ holds.

In Figure~\ref{fig:Stability} a reader can find a graphical overview of the results obtained in this paper, in particular, the relationship between the introduced stability notions.

%We need also the following property of Dini derivatives (see e.g. \cite[Lemma 3.4]{MiW17a} or 
%\cite[pp. 204-205]{Sak47}, \cite{HaT06}).
%
%\begin{lemma}
%\label{Dinilemma}
      %Let $f,g: [0,\infty) \to \R$ be continuous. If for all $t\geq 0$ we
      %have $D_+ f(t) \leq - g(t)$, then for all $t\geq 0$ it follows that
%\begin{equation}
    %\label{eq:33}
    %f(t) - f(0) \leq - G(t) := - \int_0^t g(s) ds.
%\end{equation}
%\end{lemma}

\section{Criteria for iUGATT and iUGAS}
\label{sec:iUGAS_criterion_and_Noncoercivity}

In this section we study ``integral'' stability properties starting with
criteria for integral UGATT and then for iUGAS.

\subsection{Criteria for integral UGATT}

First we would like to give a criterion for iUGATT in terms of UGWA.
To this end we need one more concept:
\begin{definition}
\label{def:integral_Ultimate_uniform_stability} 
Let $\Sigma$ be a forward complete system. We say that the fixed point $\Sigmafp$ is ultimately (locally) integrally stable (ultimately iULS) if there is $\alpha\in\K$
so that for any $\eps>0$ there exist $T=T(\eps)>0$ and $\delta = \delta (\eps)>0$ so that 
\begin{eqnarray}
\|x\|_X \leq \delta,\ d \in \Dc \quad \Rightarrow \quad \int_{T(\eps)}^\infty \alpha\big(\|\phi(t,x,d)\|_X\big)ds \leq \eps.
\label{eq:integral_UltimateLyapunovStability_epsdelta}
\end{eqnarray}
\end{definition}

Now we are in a position to characterize iUGATT.
\begin{proposition}
\label{prop:iUGATT_equiv_UGWA_Ult_iULS}
Consider a forward complete system $\Sigma=(X,\Dc,\phi)$ with fixed point $\Sigmafp$.
Then $\Sigmafp$ is iUGATT with some $\alpha\in\K$ if and only if $\Sigmafp$ is UGWA and ultimately iULS with the same $\alpha$.
\end{proposition}

\begin{proof}
$\Rightarrow$. Assume $\Sigmafp$ is iUGATT for a given $\alpha\in\K$. Ultimate iULS of $\Sigmafp$ (with the same weight function $\alpha$) easily follows from iUGATT. Let us show that $\Sigmafp$ is UGWA.

Pick any $R>0$ and any $\eps>0$. Since $\Sigmafp$ is integrally UGATT, there is a time $\tilde\tau=\tilde\tau(R,\eps)$ so that 
\begin{eqnarray*}
\sup_{x\in \overline{B_r},\;d\in\Dc} \int_{\tilde{\tau}(R,\eps)}^\infty \alpha(\|\phi(s,x,d)\|_X)ds  \leq \frac{1}{2}\alpha(\eps).
\end{eqnarray*}

Assume that for some $x\in \overline{B_r}$, some $d\in\Dc$ and all $s\in[\tilde{\tau}(R,\eps),\tilde{\tau}(R,\eps)+1]$ it holds that $\|\phi(s,x,d)\|_X\geq \eps$. Then 
\begin{eqnarray*}
\frac{1}{2}\alpha(\eps) \geq \int_{\tilde{\tau}(R,\eps)}^\infty
\alpha(\|\phi(s,x,d)\|_X)ds %\\ \geq 
%\int_{\tilde{\tau}(R,\eps)}^{\tilde{\tau}(R,\eps)+1} \alpha(\|\phi(s,x,d)\|_X)ds
\geq \alpha(\eps),
\end{eqnarray*}
a contradiction.
This shows that $\Sigmafp$ is uniformly globally weakly attractive with $\tau(R,\eps):=\tilde{\tau}(R,\eps) + 1$.
 
$\Leftarrow$. Since $\Sigmafp$ is ultimately iULS, there exists $\alpha\in\K$ so that for all $\eps>0$ there are $\delta(\eps)>0$ and $T(\eps)>0$ so that \eqref{eq:integral_UltimateLyapunovStability_epsdelta} holds.

Pick any $\varepsilon>0$ and $r>0$. Since $\Sigmafp$ is uniformly globally weakly attractive, there is a time $\tilde\tau=\tilde\tau(r,\eps)$ so that for any $x\in B_r$ and any $d\in \Dc$ there is a time $\bar{t}\in [0,\tilde\tau(r,\eps))$ so that $\|\phi(\bar{t},x,d)\|_X\leq \delta(\eps)$.

In view of the ultimate iULS property we have that 
\[
t\geq T(\eps) \srs \int_t^\infty \alpha\big(\|\phi(s,\phi(\bar{t},x,d),d(\bar{t}+\cdot))\|_X\big)ds \leq \eps.
\]    
Due to the cocycle property
%\begin{eqnarray*}
%\phi(t+\bar{t},x,d)= \phi(t,\phi(\bar{t},x,d),d(\bar{t} +\cdot)),
%\end{eqnarray*}
it holds that
\begin{multline*}
 \int_t^\infty \alpha\big(\|\phi(s,\phi(\bar{t},x,d),d(\bar{t}+\cdot))\|_X\big)ds \\ 
=
\int_t^\infty \alpha\big(\|\phi(s+\bar{t},x,d)\|_X\big)ds  
=
\int_{t+\bar{t}}^\infty \alpha\big(\|\phi(s,x,d)\|_X\big)ds.
\end{multline*}
%
%Summarizing the above derivations, we have that 
%\[
%t\geq \bar{t} + T(\eps),\ x\in B_r,\ d\in\Dc \srs \int_t^\infty \alpha\big(\|\phi(s,x,d)\|_X\big)ds \leq \eps,
%\]    
%Since $\bar{t}\leq \tau(r,\eps)$ for all $x\in B_r$ and any $d\in \Dc$, we have 
%\[
%t\geq \tau(r,\eps) + T(\eps),\ x\in B_r,\ d\in\Dc \srs \int_t^\infty \alpha\big(\|\phi(s,x,d)\|_X\big)ds \leq \eps.
%\]    
%Specializing this to $t\geq \tilde\tau(r,\eps)$, we 
Considering $t\geq \bar{t} + T(\eps)$, it is now easy to see that $\Sigmafp$ is iUGATT (with the same $\alpha$).
\end{proof}

Analogously to Proposition~\ref{prop:iUGATT_equiv_UGWA_Ult_iULS} one can
characterize the UGATT property. We define
% \begin{definition}
% \label{def:Ultimate_uniform_stability} 
% A nonempty set $\A\subset X$ is called ultimately uniformly stable if for any $\eps>0$ there exist $T=T(\eps)$ and $\delta = \delta (\eps)$ so that 
% \begin{eqnarray}
% t \geq T,\ \|x\|_\A \leq \delta,\ d \in \Dc \quad \Rightarrow \quad \|\phi(t,x,d)\|_\A \leq \eps.
% \label{eq:UltimateLyapunovStability_epsdelta}
% \end{eqnarray}
% \end{definition}
\begin{definition}
\label{def:Ultimate_uniform_stability} 
Consider a forward complete system $\Sigma=(X,\Dc,\phi)$ with fixed point $\Sigmafp$.
The fixed point $\Sigmafp$ is called ultimately uniformly stable if for any $\eps>0$ there exist $T=T(\eps)>0$ and $\delta = \delta (\eps)>0$ so that 
\begin{eqnarray}
t \geq T,\ \|x\| \leq \delta,\ d \in \Dc \quad \Rightarrow \quad \|\phi(t,x,d)\| \leq \eps.
\label{eq:UltimateLyapunovStability_epsdelta}
\end{eqnarray}
\end{definition}

\begin{proposition}
\label{prop:UGATT_equiv_UGWA_Ult_ULS}
Consider a forward complete system $\Sigma=(X,\Dc,\phi)$ with fixed point $\Sigmafp$.
%Then $\Sigmafp$ is a UGATT set if and only if $\Sigmafp$ is ultimately uniformly stable and UGWA.
Then $\Sigmafp$ is UGATT if and only if $\Sigmafp$ is ultimately uniformly stable and UGWA.
\end{proposition}

\begin{proof}
"$\Rightarrow$". Clear.

"$\Leftarrow$". Fix $\eps>0$. Since $\A$ is ultimately uniformly stable,
there are positive $\delta(\eps)$ and $T(\eps)$ so that 
\eqref{eq:UltimateLyapunovStability_epsdelta} holds.
Now pick any $r>0$.
% Using the cocycle property and uniform global weak attractivity of
% $\Sigmafp$ for the value $\eps$ and the time instant $T(\eps)$, we see that
% %, according to Lemma~\ref{lem:unif_weak_attr_criterion} 
% there is a time $\tilde\tau=\tilde\tau(r,\eps)$ so that for any $x\in B_r(\A)$ and any $d\in \Dc$ there is a time $\bar{t}\in [T(\eps),T(\eps) + \tilde\tau(r,\eps))$ so that $\|\phi(\bar{t},x,d)\|_\A\leq \delta(\eps)$.
By uniform global weak attractivity of
$\Sigmafp$ 
%, according to Lemma~\ref{lem:unif_weak_attr_criterion} 
there is a time $\tilde\tau=\tilde\tau(r,\eps)$ so that for any $x\in B_r(\A)$ and any $d\in \Dc$ there is a time $\bar{t}\in [0, \tilde\tau(r,\eps))$ so that $\|\phi(\bar{t},x,d)\|\leq \delta(\eps)$.

Due to the cocycle property
\begin{eqnarray*}
\phi(t+\bar{t},x,d)= \phi(t,\phi(\bar{t},x,d),d(\bar{t} +\cdot)),
\end{eqnarray*}
and in view of ultimate uniform stability we have that 
\[
t\geq \bar{t}+T(\eps),\ x\in B_r(\A),\ d\in\Dc \srs \|\phi(t,x,d)\|\leq \eps,
\]    
Specializing this to $t\geq \tilde\tau(r,\eps)+T(\eps)$, we see that $\A$ is UGATT.
\end{proof}

\subsection{Characterization of iUGAS}

In \cite[Proposition 3.7]{MiW17a} the following result has been obtained (the statement in \cite{MiW17a} was somewhat different, but the proof is exactly the same):
\begin{proposition}
\label{prop:From_ncLF_to_UGWA}
Consider a forward complete system $\Sigma=(X,\Dc,\phi)$  with fixed point $\Sigmafp$. If $\Sigmafp$ is iUGS, then $\Sigmafp$ is UGWA.
\end{proposition}

We also note:
\begin{lemma}
\label{lem:iULS_criterion} 
Consider a forward complete system $\Sigma=(X,\Dc,\phi)$ with fixed point $\Sigmafp$.
Then $\Sigmafp$ is iULS if and only if $\Sigmafp$ is an iREP and ultimately iULS.
\end{lemma}

\begin{proof}
$\Rightarrow$. This is clear.

$\Leftarrow$. Since $\Sigmafp$ is ultimately iULS, there is $\alpha_1\in\K$ so that for any $\varepsilon>0$ there are 
$r=r(\eps)>0$ and a time $\tau=\tau(\varepsilon)>0$ satisfying 
\begin{eqnarray*}
\|x\|_X\leq r(\eps),\ d \in\Dc \srs \int_{\tau}^\infty \alpha_1\big(\|\phi(s,x,d)\|_X\big)ds \leq \frac{\varepsilon}{2}. 
\end{eqnarray*}
Now since $\Sigmafp$ is an iREP,  there is $\alpha_2\in\K$ so that  for these $\varepsilon,\tau$ there is a $0<\tilde\delta=\tilde{\delta}(\varepsilon)\leq r(\eps)$ so that
\begin{eqnarray*}
\|x\|_X\leq \tilde\delta, \ d \in\Dc \srs \int_0^{\tau} \alpha_2\big(\|\phi(s,x,d)\|_X\big)ds \leq \frac{\varepsilon}{2}. 
\end{eqnarray*}
Define $\alpha(s):=\min\{\alpha_1(s),\alpha_2(s)\}$, $s\geq0$. Clearly, $\alpha\in\K$ and it holds that
\begin{eqnarray*}
\|x\|_X\leq \tilde\delta, \ d \in\Dc \srs \int_0^{+\infty} \alpha\big(\|\phi(s,x,d)\|_X\big)ds \leq \varepsilon.
\end{eqnarray*}
Without loss of generality we can assume that $\tilde{\delta}$ is
non-decreasing as a function of $\eps$. Furthermore, by construction it
holds that $\tilde{\delta}$ can be continuously extended by $\tilde{\delta}(0)=0$. Then it can be lowerbounded by a certain $\delta\in\K$.

Now iULS of $\Sigmafp$ follows by choosing $\psi(s):=\delta^{-1}(s)$,
$s\in[0, \lim_{s\to\infty} \delta(s))$.
\end{proof}

% {\color{red}
% Should one delete the criteria (v) and (vi) in order to make the result more coincise?
% }

The main result in this section is the characterization of the notion of iUGAS:
\begin{theorem}
\label{thm:iUGAS_characterization} 
Consider a forward complete system $\Sigma=(X,\Dc,\phi)$. Then the following statements are equivalent:
\begin{enumerate}[(i)]
	\item $\Sigmafp$ is iUGAS.
	\item $\Sigmafp$ is iUGS.
	\item \label{item3} $\Sigmafp$ is iULS (with a certain $\alpha\in\K$) and $\Sigmafp$ is UGWA.
	\item \label{item4} $\Sigmafp$ is an iREP and $\Sigmafp$ is iUGATT.
	\item $\Sigmafp$ is iULS and $\Sigmafp$ is iUGATT.
	\item $\Sigmafp$ is iUGS and $\Sigmafp$ is iUGATT.
\end{enumerate}
Moreover, in item~(\ref{item4}) the function $\alpha$ can be chosen to be equal
to $\alpha$ from item~(\ref{item3}).
\end{theorem}

\begin{proof}
\textbf{(i) $\Rightarrow$ (ii).} Evident.

\textbf{(ii) $\Rightarrow$ (iii).} Follows by Proposition~\ref{prop:From_ncLF_to_UGWA}.

\textbf{(iii) $\Rightarrow$ (iv).} 
As $\Sigmafp$ is iULS it follows that it is an iREP.
Furthermore, since $0$ is UGWA and ultimately iULS with $\alpha\in\K$, then, by means of 
Proposition~\ref{prop:iUGATT_equiv_UGWA_Ult_iULS}, $0$ is iUGATT with the same $\alpha$.

%Let $0$ be iULS with corresponding $\alpha\in\K$, $\psi\in\Kinf$ and $r>0$.
%Since $0$ is UGWA, for any $R>0$ and any $\eps\in(0,r)$ there is a time $\tau=\tau(R,\eps)$ so that 
%\begin{eqnarray*}
%\|x\|_X\leq R,\ d\in\Dc \srs \exists t_\eps=t_\eps(x,d)\leq\tau: \ \|\phi(t_\eps,x,d)\|_X\leq\eps.
%%\label{eq:}
%\end{eqnarray*}
%Define $y_\eps=y_\eps(x,d):=\phi(t_\eps(x,d),x,d)$. Since $\varepsilon<r$, applying \eqref{eq:Integral_boundedness} to $y_\eps$ we obtain:
%\begin{eqnarray*}
%\int_0^\infty \alpha\big(\|\phi(s,y_\eps,d(t_\eps + \cdot))\|_X\big)ds \leq \psi(\|y_\eps\|_X) = \psi(\eps).
%%\label{eq:}
%\end{eqnarray*}
%Due to cocycle property we know that $\phi(s,y_\eps,d(t_\eps + \cdot))=\phi(t_\eps + s,x,d)$, for any $s\geq0$, which leads to the following equalities:
%\begin{multline*}
%\hspace*{-0.35cm}\int_0^\infty \alpha\big(\|\phi(s,y_\eps,d(t_\eps + \cdot))\|_X\big)ds = 
%\int_0^\infty \alpha\big(\|\phi(t_\eps + s,x,d)\|_X\big)ds \\
 %= 
%\int_{t_\eps}^\infty \alpha\big(\|\phi(s,x,d)\|_X\big)ds. 
%\end{multline*}
%Since $t_\eps(x,d)\leq\tau(R,\eps)$ for any $x\in \overline{B_R}$, any $d\in\Dc$,
%we obtain that 
%\begin{eqnarray*}
%\sup_{x\in\overline{B_R},\;d\in\Dc} \int_{\tau(r,\eps)}^\infty \alpha(\|\phi(s,x,d)\|_X)ds
%\leq\psi(\eps).
%%\label{eq:}
%\end{eqnarray*}
%Taking the limit $\eps\to+0$ and recalling that $\psi\in\Kinf$, we see that $\Sigmafp$ is iUGATT.

%\textbf{(iv) $\Rightarrow$ (i).} This part is omitted for reasons of
%space.
%\vspace*{-2.5ex}

\textbf{(iv) $\Rightarrow$ (v).} This follows directly from
Lemma~\ref{lem:iULS_criterion} and Proposition~\ref{prop:iUGATT_equiv_UGWA_Ult_iULS}.

\textbf{(v) $\Rightarrow$ (vi).}
Note that if $\Sigmafp$ is iULS or iUGATT for a certain $\alpha\in\K$,
then the same is true for any $\tilde{\alpha} \in {\cal K}$ with
$\tilde{\alpha}\leq\alpha$ (here $\leq$ is the pointwise ordering). 
Thus without loss of generality $\Sigmafp$ is both iULS and iUGATT for same $\alpha\in\K\backslash\Kinf$.

As $\Sigmafp$ is iUGATT we know that for each $R>0$ and $\eps>0$ there is $\tau:=\tau(R,\eps)$ so that
\begin{eqnarray*}
\sup_{x\in \overline{B_R},\;d\in\Dc} \int_\tau^\infty \alpha(\|\phi(s,x,d)\|_X)ds \leq \eps.
\end{eqnarray*}
Denoting $\alpha(\infty):=\lim_{s\to\infty}\alpha(s) <\infty$ we see that
\begin{eqnarray*}
\sup_{x\in \overline{B_R},\;d\in\Dc} \int_0^\tau \alpha(\|\phi(s,x,d)\|_X)ds \leq \tau \alpha(\infty),
\end{eqnarray*}
% which means that $\Sigma$ is $\alpha$-integrally RFC with the above $\alpha\in\K\backslash\Kinf$. This shows that
and hence, for all $R>0$,
\begin{eqnarray*}
\tilde{\sigma}(R):=  \sup_{x\in \overline{B_R},\;d\in\Dc} \int_0^\infty \alpha(\|\phi(s,x,d)\|_X)ds <\infty.
\end{eqnarray*}
Clearly, $\tilde{\sigma}$ is a non-decreasing function of $R$, and so
there exist  $\overline{\sigma}\in\Kinf$ and $c>0$ so that
$\tilde{\sigma}(r)\leq \overline{\sigma}(r) + c$ for all $r\in\R_+$.
Consequently, for any $x\in X$ and any $d\in\Dc$ we obtain
\begin{eqnarray}
\label{eq:practical_iUGS_estimate_tmp}
\int_0^\infty \alpha(\|\phi(s,x,d)\|_X)ds   \leq \overline{\sigma}(\|x\|_X) + c.
\end{eqnarray}
As $\Sigmafp$ is iULS, there exist $\psi\in\Kinf$ and $r>0$ so that
\begin{eqnarray}
\label{eq:_iULS_estimate_tmp}
\|x\|_X\leq r,\ d\in\Dc \srs \int_0^\infty \alpha(\|\phi(s,x,d)\|_X)ds   \leq \psi(\|x\|_X).
\end{eqnarray}
Using \eqref{eq:practical_iUGS_estimate_tmp},
\eqref{eq:_iULS_estimate_tmp} and standard manipulations of
$\Kinf$-functions (see e.g. \cite[proof of Lemma I.2, p. 1287]{SoW96}), it
may be seen that there is a $\sigma\in\Kinf$ so that 
\begin{eqnarray*}
x \in X,\ d\in\Dc \srs \int_0^\infty \alpha(\|\phi(s,x,d)\|_X)ds  \leq \sigma(\|x\|_X).
\end{eqnarray*}
This shows that $\Sigmafp$ is iUGS with $\alpha\in\K\backslash\Kinf$.

 \textbf{(vi) $\Rightarrow$ (i).}
 As in the previous step, without loss of generality we may assume that the function $\alpha$
 is the same in the definitions of iUGS and iUGATT. We now consider a
 fixed, suitable
 $\alpha$.

 Since $\Sigmafp$ is a iUGS fixed point, there exists $\psi\in\Kinf$ so that for all
 $t\geq0$, $\delta\geq0$, $\|x\|_X \leq \delta$, $d \in \Dc$ we have
 \begin{eqnarray}
 %\hspace{-9mm}t\geq 0,\, \|x\|_X \leq \delta,\, d \in \Dc \srs
 %\int_0^\infty \alpha(\|\phi(t,x,d)\|_X) dt \leq \psi(\delta).
 \int_0^\infty \alpha(\|\phi(t,x,d)\|_X) dt \leq \psi(\delta).
 \label{eq:TmpEstimate}
 \end{eqnarray}
 For a fixed $\delta\geq0$, define
 $\eps_n:=\frac{1}{2^n} \psi(\delta)$, $n \in \N$. 
 Let $\tau_0:=0$.
 As $\Sigmafp$ is iUGATT there exist times $\tau_n:=\tau(\eps_n,\delta)$,
 $n\geq1$, which we assume without loss of generality to be strictly increasing, such that 
 \[
 t \geq \tau_n,\ \|x\|_X \leq \delta,\ d \in \Dc \srs \int_t^\infty \alpha(\|\phi(s,x,d)\|_X) ds \leq \eps_n.
 \]
 Define $\omega(\delta,0):= 2 \psi(\delta)$ and $\omega(\delta,\tau_n):=\eps_{n-1}$, for $n \in \N$, $n \neq 0$. Extend the function $\omega(\delta,\cdot)$ to $\R_+$ so that $\omega(\delta,\cdot) \in \LL$. 
 Note that for any $n\in\N$ and for all $t \in (\tau_n,\tau_{n+1})$ it holds that 
 \begin{eqnarray}
 \int_t^\infty \alpha(\|\phi(s,x,d)\|_X) ds \leq \eps_n < \omega(\delta,t).
 \label{eq:Tmp_Estimate}
 \end{eqnarray}
 Doing this for all $\delta \in \R_+$ we obtain a function $\omega: \R_+^2\to\R_+$.

 Now define $\tilde\beta(r,t):=\sup_{0 \leq s \leq r}\omega(s,t) \geq \omega(r,t)$. Obviously, $\tilde\beta$ is non-decreasing in the first argument and decreasing in the second. 
 Moreover, for each fixed $t\geq0$, $\tilde \beta(r,t) \leq \sup_{0 \leq s \leq r}\omega(s,0)=2\psi(r)$, which implies that $\tilde\beta$ is continuous in the first argument at $r=0$ for any fixed $t\geq0$.
 Now Proposition~\ref{prop:Upper_estimate_as_KL_function} implies that $(r,t) \mapsto \tilde\beta(r,t) + |r|e^{-t}$ may be upper bounded by $\beta\in \KL$ and the estimate 
 \eqref{eq:integral_UGAS}
 is satisfied with such a $\beta$.
\end{proof}

\begin{remark}
\label{rem:integral_notions_with_Kinf_functions} 
Note that in all the integral notions we have assumed that the corresponding function $\alpha$ belongs to the class $\K$.
If we require in the definitions that $\alpha$ must belong to the class
$\Kinf$, we obtain stronger versions of the corresponding concepts.
The difference is that every forward-complete system is automatically
$\alpha$-integrally RFC with $\alpha\in\K\backslash\Kinf$ (see Remark~\ref{rem:RFC_and_iRFC}), 
but it need not be $\alpha$-integrally RFC for all $\alpha\in\Kinf$.

For the stronger concepts the proof of (iv) $\Rightarrow$ (i) in
Theorem~\ref{thm:iUGAS_characterization} does not work as described. 
In order to close the gap in the proof the stronger version of this implication, we need to
strengthen the assumptions in items
(iii), (iv) by assuming in addition 
that the system is  $\alpha$-iRFC with a certain $\alpha\in\Kinf$.
Then after some minor modifications we recover the characterization of
iUGAS with $\alpha\in\Kinf$.

We leave the details to the reader.
\end{remark}

\section{``Integral'' characterization of the UGAS property}
\label{sec:UGAS_characterizations}

Until now we have worked nearly completely on the level of the ``integral''
notions, which is almost parallel to the world of classic notions of stability. Now we are going to relate ``integral'' and ``classic'' worlds.

The next proposition shows that classic stability properties can be
recovered from the ``integral'' version combined with the REP property.

\begin{proposition}
\label{prop:Influence_of_REP_and_RFC}
Consider a forward complete system $\Sigma=(X,\Dc,\phi)$. Then the following holds:
\begin{itemize}
	\item[(i)] \label{prop:From_ncLF_to_LS} If $\Sigmafp$ is a REP and iULS, then $\Sigmafp$ is ULS.
	\item[(ii)] \label{prop:NewProp_and_REP_imply_UGATT} 
	If $\Sigmafp$ is a REP and iUGATT, then $\Sigmafp$ is UGATT and UAS.
%FW: the following is now in the next remark
% {\color{blue}
% 	\item[(iii)] \label{cor:RFC_plus_integral_notions_imply_pUGAS} 
% 	If $\Sigma$ is RFC and $\Sigmafp$ is iUGATT or iUGS or iUGAS, then $\Sigma$ is practically UGAS.
% }	
\end{itemize}
\end{proposition}

\begin{remark}
    We note that in \cite[Theorem 3.1]{Mir17a} it is shown that if
    $\Sigma$ is RFC and $\Sigmafp$ is iUGATT or iUGS or iUGAS, then
    $\Sigma$ satisfies a property that is termed practically UGAS in
    \cite{Mir17a} and which amounts to saying that not the fixed point
    $\Sigmafp$ but a certain neighborhood of it has a stability property.
\end{remark}

\begin{proof} (of Proposition~\ref{prop:Influence_of_REP_and_RFC}).

\textbf{(i).}
Seeking a contradiction, assume that $\Sigma$ is not uniformly stable in $x^* = 0$. Then there exist an
$\varepsilon >0$ and sequences $\{ x_k \}_{k\in\N}$ in $X$, $\{ d_k
\}_{k\in\N}$ in $\mathcal{D}$, and $t_k \geq 0$ such that $x_k \to 0$ as $k \to \infty$ and
\begin{equation*}
    \| \phi(t_k,x_k,d_k) \|_X = \varepsilon \quad \forall k \geq 1.
\end{equation*}

Since $\Sigmafp$ is iULS, there are $\alpha\in\K$ and $\psi\in\Kinf$ so that for the above $\eps$ there is a $\delta_1=\delta_1(\varepsilon)>0$ satisfying
\begin{eqnarray}
\|x\|_X\leq \delta_1,\ d\in\Dc \srs \int_0^\infty \alpha\big(\|\phi(s,x,d\|_X\big)ds \leq \psi(\|x\|_X).
\label{eq:iULS_tmp}
\end{eqnarray}
Without loss of generality we assume that $\|x_k\|_X\leq \delta_1$ for all $k\in\N$ (otherwise we can pick a subsequence of $\{x_k\}$ with this property).

Since $0$ is a REP, for the above $\eps$ there is a $\delta = \delta(\eps,1)$ so that
\begin{eqnarray}
\|x\|_X\leq \delta,\ t\in [0,1],\ d\in\Dc \qrq \|\phi(t,x,d)\|_X\leq \frac{\eps}{2}.
\label{eq:ncLF_REP_application}
\end{eqnarray}
Define for this $\delta$ the following quantities:
\[
\tilde t_k := \sup\{t\in[0,t_k]: \|\phi(t,x_k,d_k)\|_X \leq \delta\},
\]
provided the supremum is taken over a nonempty set, and $\tilde{t}_k:=0$ otherwise.
Denote also $\eta_k:=t_k - \tilde t_k$, $k\in\N$. There are two possibilities.

First assume that $\{\eta_k\}_{k\in\N}$ does not converge to 0 as $k\to\infty$.
Then there is a $\eta^*>0$ and a subsequence of $\{\eta_{k_m}\}$ so that $\eta_{k_m} \geq \eta^*$ for all $m\geq 1$.

Using \eqref{eq:iULS_tmp} for $x:=x_{k_m}$, $d:=d_{k_m}$ and $t:=t_{k_m}$, we see that
\begin{eqnarray*}
\eta^* \alpha(\delta) \leq \eta_{k_m} \alpha(\delta) 
%&\leq& \int_{\tilde t_{k_m}}^{t_{k_m}} \alpha(\|\phi(s,x_{k_m},d_{k_m})\|_X)ds \\
%																		&\leq &\int_0^{t_{k_m}} \alpha(\|\phi(s,x_{k_m},d_{k_m})\|_X)ds  
\leq \psi(\|x_{k_m}\|_X).
\end{eqnarray*}
Since $\psi(\|x_{k_m}\|_X)\to0$ as $m\to\infty$, we obtain a contradiction.

Now assume that $\eta_k\to 0$ as $k\to\infty$. Then there is a $k_1>0$ so that $\eta_{k_1}<1$.
In view of a cocycle property, we have that
\begin{eqnarray*}
\phi(t_k,x_k,d_k) = \phi(\eta_k,\phi(\tilde t_k,x_k,d_k),d_k(\cdot + \tilde t_k)).
\end{eqnarray*}
Since $\|\phi(\tilde t_k,x_k,d_k)\|_X \leq \delta$, by
\eqref{eq:ncLF_REP_application} we obtain % that 
$\|\phi(t_k,x_k,d_k)\|_X \leq  \frac{\eps}{2}$,
%\label{eq:}
which contradicts to the assumption that $\|\phi(t_k,x_k,d_k)\|_X = \eps$.
This shows uniform stability of $\Sigmafp$.

\textbf{(ii).}
It is easy to see that iUGATT implies ultimate iULS. According to
Lemma~\ref{lem:REP_implies_iREP}, $\Sigmafp$ is an iREP.
Using Lemma~\ref{lem:iULS_criterion} % the equilibrium $\Sigmafp$ is iULS.
%Now, by 
and Proposition~\ref{prop:Influence_of_REP_and_RFC}~(i) we have that $\Sigmafp$ is ULS.

Furthermore, by Proposition~\ref{prop:iUGATT_equiv_UGWA_Ult_iULS}  the equilibrium point $\Sigmafp$ is UGWA, and Proposition~\ref{prop:UGATT_equiv_UGWA_Ult_ULS}
shows that $\Sigmafp$ is UGATT.
Finally, since $\Sigmafp$ is UGATT and ULS, then $\Sigmafp$ is UAS as well.
%
% {\color{blue}
% \textbf{(iii).} Follows from \cite[Theorem 3.1]{Mir17a}.
% }
\end{proof}

Next we show criteria for UGAS in terms of integral stability
notions. To this end we need two technical results. The first one is
Sontag's well-known $\KL$-lemma \cite[Proposition 7]{Son98}:
\begin{lemma}
\label{Sontags_KL_Lemma}
For all $\beta \in \KL$ there exist $\alpha_1,\alpha_2 \in \Kinf$ with
\begin{equation}
\beta(r,t) \leq \alpha_2(\alpha_1(r)e^{-t}) \quad \forall r \geq 0, \; \forall t \geq 0.
\label{eq:KL-Lemma_Estimate}
\end{equation}
\end{lemma}
The second one is a characterization of UGAS in terms of the UGATT property from \cite[Theorem 2.2]{KaJ11b}:
\begin{proposition}
\label{prop:UGAS_Characterization}
Consider $\Sigma=(X,\Dc,\phi)$.
Then $\Sigmafp$ is UGAS if and only if $\Sigma$ is robustly forward complete and
$0$ is a UGATT robust equilibrium point for $\Sigma$.
\end{proposition}

The main result of this section is:
\begin{theorem}
\label{thm:UGAS_characterization} 
Consider a forward complete system $\Sigma=(X,\Dc,\phi)$. Then the following statements are equivalent:
\begin{enumerate}[(i)]
	\item $\Sigmafp$ is UGAS.
	\item $\Sigma$ is RFC and $\Sigmafp$ is a REP $\wedge$ iUGAS.
	\item $\Sigma$ is RFC and $\Sigmafp$ is a REP $\wedge$ iUGATT.
	\item $\Sigma$ is RFC and $\Sigmafp$ is a REP $\wedge$ UGWA $\wedge$ ultimately iULS.
	\item $\Sigma$ is RFC and $\Sigmafp$ is a REP $\wedge$ UGWA $\wedge$ ultimately ULS.
\end{enumerate}
\end{theorem}

\begin{proof}
\textbf{(i) $\Rightarrow$ (ii).} 
Since $\Sigmafp$ is UGAS, there is a $\beta\in\KL$ so that \eqref{UGAS_wrt_D_estimate} holds.
In view of Lemma~\ref{Sontags_KL_Lemma} there are $\alpha_1,\alpha_2\in\Kinf$ so that 
\eqref{eq:KL-Lemma_Estimate} holds.
Set $\alpha:=\alpha^{-1}_2$. Then we have for any $r>0$ and any $t>0$ it holds that:
\begin{multline*}
\hspace*{-0.3cm}\sup_{x\in \overline{B_r},\ d\in\Dc} \int_t^\infty \alpha(\|\phi(s,x,d)\|_X)ds  \leq
\sup_{x\in \overline{B_r}} \int_t^\infty \alpha(\beta(\|x\|_X,s))ds  \\
\leq  \int_t^\infty \alpha_1(r)e^{-s}ds 
= \alpha_1(r)e^{-t}
\end{multline*}
and $\Sigmafp$ is iUGAS with $\psi:=\alpha_1 \in\Kinf$ and $\alpha\in\Kinf$.

\textbf{(ii) $\Rightarrow$ (iii).} Clear.

\textbf{(iii) $\Iff$ (iv).} Follows from Proposition~\ref{prop:iUGATT_equiv_UGWA_Ult_iULS}

\textbf{(iii) $\Rightarrow$ (v).} Follows from Proposition~\ref{prop:Influence_of_REP_and_RFC}, item (ii).

\textbf{(v) $\Rightarrow$ (i).} Follows from Propositions~\ref{prop:UGATT_equiv_UGWA_Ult_ULS} and \ref{prop:UGAS_Characterization}.
\end{proof}

\begin{remark}[``Atomic decompositions'']
\label{rem:} 
Items (iv) and (v) of Theorem~\ref{thm:UGAS_characterization} give a decomposition of UGAS into elementary stability notions. In some sense the notions of UGWA, REP, RFC and ultimate ULS and their integral counterparts iREP, ultimate iULS and (possibly) iRFC are the ``atoms'' by combinations of which the other stability notions can be constructed.

Comparing items (iv) and (v) of Theorem~\ref{thm:UGAS_characterization} to
the analogous ``atomic'' decompositions of iUGAS shown in
Theorem~\ref{thm:iUGAS_characterization}, we see that the notion of UGWA plays a remarkable role in such characterizations, supported by the integral variants of REP and ultimate ULS.
Uniform global weak attractivity is the common point of the worlds of classic and integral notions, which are otherwise largely parallel.
\end{remark}

\begin{remark}
%\label{rem:} 
It is worth mentioning that for the special case of linear systems over Banach spaces without disturbances the notions of UGAS, iUGAS and UGWA  coincide, as can be seen from \cite[Proposition 5.1]{Mir17a}.
\end{remark}

\section{Non-coercive Lyapunov theorems}
\label{sec:Non-coercive_LFs}

In this section we relate the existence of noncoercive Lyapunov functions
to the integral stability concepts we have introduced. It is shown that
for forward complete systems the existence of noncoercive Lyapunov
functions implies iUGAS.
% and iUGATT. 
In the next step we treat a converse
result. 

\subsection{Direct Lyapunov theorems}

For the proof of direct Lyapunov theorems we need the generalized Newton-Leibniz formula (see \cite[Theorem 7.3, p. 204-205]{Sak47} and the comments directly after that result):
\begin{proposition}
\label{prop:Generalized_Newton-Leibniz_from_Saks_p204} 
Suppose that $F:\R \to \R$ is a function\footnote{In the formulation of \cite[Theorem 7.3, p. 204-205]{Sak47} the terminology that $F$ is a finite function is used, which means that $F(x) \in\R$ for any $x\in\R$ (see \cite[p. 6]{Sak47}).} such that for all $x\in\R$ we have
\begin{eqnarray}
\mathop{\overline{\lim}} \limits_{h \rightarrow +0} F(x-h) \leq F(x) \leq 
\mathop{\overline{\lim}} \limits_{h \rightarrow +0} F(x+h).
\label{eq:Nice_Discontinuity}
\end{eqnarray}
Let $g$ be a Perron-integrable\footnote{For a definition of Perron
  integrability see e.g. \cite[p. 201]{Sak47}. The (P) in front of the
  integral in \eqref{eq:Generalized_Newton-Leibniz} indicates that this is
a Perron integral.} function of a real variable satisfying 
$D^+F(x) \geq g(x)$ for all $x\in I$.
%, and that, further, we have: 
%\begin{itemize}
	%\item[(i)] $D^+F(x) \geq g(x)$ for all $x\in I$ 
	%%\item[(ii)] $F^+(x) >-\infty$ for all points $x\in I$ except possibly a countable set. 
%\end{itemize}
Then for all $a,b >0$: $a<b$ it holds that 
\begin{eqnarray}
F(b) - F(a) \geq (P)\int_a^b g(x) dx.
\label{eq:Generalized_Newton-Leibniz}
\end{eqnarray}
\end{proposition}

Using Theorem~\ref{thm:iUGAS_characterization} and Proposition~\ref{prop:Generalized_Newton-Leibniz_from_Saks_p204},
we can show that the existence of a non-coercive Lyapunov function implies iUGAS without any
further requirements on the flow of the system. If we
additionally assume either the REP or the RFC property, we obtain additional stability properties.

\begin{theorem}
\label{thm:direct_Non-coercive_Lyapunov_theorem} 
Consider a forward complete system $\Sigma=(X,\Dc,\phi)$. 
Assume that $V$ is a non-coercive Lyapunov function for $\Sigma$ with  corresponding functions $\alpha\in\K$ and $\psi_2\in\Kinf$. Then:
\begin{enumerate}[(i)] 
	\item The following statements hold:
\begin{itemize}%[label=(\roman*)]%[(i)]
	\item[(i-a)] $\Sigmafp$ is iUGS with this $\alpha$ and with $\psi:=\psi_2$.
	\item[(i-b)] $\Sigmafp$ is iUGATT with this $\alpha$.
	\item[(i-c)] $\Sigmafp$ is iUGAS.
\end{itemize}
%	\item If additionally $\Sigma$ is RFC, then $\Sigma$ is pUGAS
	\item If additionally $\Sigmafp$ is a REP, then $\Sigmafp$ is UGATT and UAS.
	\item If additionally $\Sigmafp$ is a REP and $\Sigma$ is RFC, then $\Sigmafp$ is UGAS.
\end{enumerate}
\end{theorem}

\begin{proof} 
\textbf{(i-a).} 
Since $V$ is a non-coercive Lyapunov function (with a corresponding $\alpha\in\K$), we have the decay estimate \eqref{DissipationIneq_UGAS_With_Disturbances}.
Pick any $x\in X$ and any $d\in\Dc$ and define $\xi :\R_+ \to\R$ via $\xi(t):=-V(\phi(t,x,d))$.

Along the trajectory $\phi$ of $\Sigma$ we have the inequality
\begin{equation}
\dot V_{d(t+\cdot)}(\phi(t,x,d)) \leq -\alpha(\|\phi(t,x,d)\|_X), \quad
\forall t\geq 0.
\label{eq:JustEq1}
\end{equation}

Due to the cocycle property we have
\begin{align*}
\dot V_{d(t+\cdot)}(&\phi(t,x,d)) \\
=& \mathop{\underline{\lim}}  \limits_{h \rightarrow +0} \frac{1}{h}\big(V(h,\phi(t,x,d),d(t+\cdot)) - V(\phi(t,x,d))\big)\\
=& \mathop{\underline{\lim}}  \limits_{h \rightarrow +0} \frac{1}{h}\big(V(\phi(t+h,x,d)) - V(\phi(t,x,d))\big)\\
=& \mathop{\underline{\lim}}  \limits_{h \rightarrow +0} \frac{1}{h}\big(-\xi(t+h) + \xi(t)\big) 
=D_+(-\xi(t))
=-D^+ \xi(t).
\end{align*}
With this new notation, equation \eqref{eq:JustEq1} can be rewritten  as 
\begin{eqnarray}
D^+ \xi(t) \geq \alpha(\|\phi(t,x,d)\|_X), \quad \forall t\geq 0.
\label{eq:JustEq2}
\end{eqnarray}
In view of \eqref{eq:V_estimate_along_trajectory_final} and since 
%\[
$\mathop{\underline{\lim}} \limits_{h \rightarrow +0}-\xi(t+h)=-\mathop{\overline{\lim}} \limits_{h \rightarrow +0}\xi(t+h)$,
%\]
we see that the inequality
\begin{eqnarray}
\mathop{\overline{\lim}} \limits_{h \rightarrow +0} \xi(t-h) \leq \xi(t) \leq 
\mathop{\overline{\lim}} \limits_{h \rightarrow +0} \xi(t+h)
\label{eq:LF_Nice_Discontinuity}
\end{eqnarray}
is satisfied for all $t>0$, and the right inequality is satisfied for $t=0$ as well.

Now we can apply
Proposition~\ref{prop:Generalized_Newton-Leibniz_from_Saks_p204} to the
above inequality. Since $t\mapsto \phi(t,x,d)$ is continuous due to the
continuity axiom $\Sigma$\ref{axiom:Continuity}, the function $g: t\mapsto
\alpha(\|\phi(t,x,d)\|_X)$ is continuous as well, and thus it is Riemann
integrable on any compact interval in $\R_+$.  As $g$ is a positive
function, the Riemann and the Perron integral
coincide (see \cite[p. 203]{Sak47}). Thus in our case the Perron integral in the
formula \eqref{eq:Generalized_Newton-Leibniz} is merely a Riemann
integral.

%
%See \cite[p. 203]{Sak47}:
%\begin{verbatim}
%\cite[Theorem 6.5, p. 203]{Sak47} shows that, although Perron integration
 %is more general than Lebesgue integration, the two processes are 
%completely equivalent in the case of integration of functions of constant sign.
%\end{verbatim}

Applying Proposition~\ref{prop:Generalized_Newton-Leibniz_from_Saks_p204}, we obtain:
\begin{eqnarray}
\xi(t) - \xi(0) \geq \int_0^t \alpha(\|\phi(s,x,d)\|_X) ds, \quad \forall t\geq 0.
\label{eq:JustEq3}
\end{eqnarray}
Since $\xi(0) = - V(\phi(0,x,d)) = -V(x)$ due to the identity axiom of
$\Sigma$, the above inequality immediately implies that
\begin{equation*}
V(\phi(t,x,d)) - V(x) \leq -\int_0^t \alpha(\|\phi(s,x,d)\|_X)ds,
%\label{eq:JustEq3}
\end{equation*}
which in turn shows that for all $t\geq 0$ we have
\begin{eqnarray}
\int_0^t \alpha(\|\phi(s,x,d)\|_X)ds  
\leq
V(x)\leq \psi_2(\|x\|_X).
\label{eq:Vintbound}
\end{eqnarray}
Taking the limit $t\to\infty$, we see that $\Sigmafp$ is iUGS.

\textbf{(i-b).} By Proposition~\ref{prop:From_ncLF_to_UGWA} and item (i) we see that $\Sigmafp$ is UGWA. Checking the proof of the Theorem~\ref{thm:iUGAS_characterization} (implication (iii) $\Rightarrow$ (iv)) we see that $\Sigmafp$ is iUGATT with the same $\alpha$.

\textbf{(i-c).} Follows from items (i), (ii) and Theorem~\ref{thm:iUGAS_characterization}.

%\textbf{(ii).} Follows by the item (i) of this theorem and by Corollary~\ref{cor:RFC_plus_integral_notions_imply_pUGAS}.

\textbf{(ii).} By the item (i-b) of this theorem, $\Sigmafp$ is iUGATT. Now Corollary~\ref{prop:NewProp_and_REP_imply_UGATT} implies that $\Sigmafp$ is UGATT and UAS.

\textbf{(iii).} By the item (i-b) of this theorem, $\Sigmafp$ is iUGATT. The rest follows from Theorem~\ref{thm:UGAS_characterization}.
\end{proof}

\begin{remark}
\label{rem:Discontinuity_of_V} 
Condition \eqref{eq:V_estimate_along_trajectory_final} means that for each $x\in X$ and $d\in \Dc$ the map 
$t\mapsto V(\phi(t,x,d))$ is either continuous at $t^*$ or the function
jumps down at $t^*$, for arbitrary $t^* \in \R_+$. 
\end{remark}

\begin{remark}
\label{rem:Relations_to_classic_results} 
The crucial difference of Theorem~\ref{thm:direct_Non-coercive_Lyapunov_theorem} from classic Lyapunov theorems is that  we do not assume the coercivity of a Lyapunov function. This makes it impossible to use any kind of a comparison principle to derive the desired UGAS stability property. 

On the other hand, in contrast to the non-coercive direct Lyapunov theorem
shown in \cite{MiW17a} we assume for item (i) of
Theorem~\ref{thm:direct_Non-coercive_Lyapunov_theorem} neither robustness
of the trivial equilibrium, nor the RFC property of the system $\Sigma$
(however, we still assume in advance the forward completeness of system
$\Sigma$). Even under such mild assumptions (and with very mild regularity
assumptions on $V$) we are able to infer the iUGAS property.
%In Section~\ref{sec:Converse_non-coercive_Lyapunov_theorem} we show that also the converse statement hold, that is 
%iUGAS of $\Sigma$ implies existence of $V$ with the properties enlisted in the statement of Theorem~\ref{thm:direct_Non-coercive_Lyapunov_theorem}.
%
%Item (ii) of Theorem~\ref{thm:direct_Non-coercive_Lyapunov_theorem} seems to be new.
We note that it is also possible to show a practical UGAS property if in
addition to the existence of $V$ we assume RFC.
Item (ii) of Theorem~\ref{thm:direct_Non-coercive_Lyapunov_theorem} is a variation of \cite[Corollary 3.10]{Mir17a} and is given here for completeness.
Item (iii) of Theorem~\ref{thm:direct_Non-coercive_Lyapunov_theorem} is
slightly stronger than \cite[Theorem 4.5]{MiW17a}, where a more direct proof of this result was given.
\end{remark}

\subsection{Converse non-coercive Lyapunov theorem}
\label{sec:Converse_non-coercive_Lyapunov_theorem}

% Theorem~\ref{thm:direct_Non-coercive_Lyapunov_theorem} shows that existence of a non-coercive Lyapunov function satisfying \eqref{eq:V_estimate_along_trajectory_final} is enough for iUGS of a forward complete system $\Sigma$.
% However, since the search of a Lyapunov function for any nontrivial system is a hard task, it is very motivating to know that such a function at least exists. 
%The next theorem shows that this is indeed the case. 

We proceed to the converse Lyapunov theorem.

\begin{theorem}
\label{thm:converse_Non-coercive_Lyapunov_theorem} 
Consider a forward complete system $\Sigma=(X,\Dc,\phi)$ and let $0$ be an equilibrium of $\Sigma$. Assume that $\Sigma$ is iUGS with $\alpha\in\K$ and $\psi\in\Kinf$.
%\begin{itemize}
	%\item[(i)] iUGS with a certain $\alpha\in\K$ and $\psi\in\Kinf$.
	%\item[(ii)] for any bounded sets $I \subset \R_+$, $S\subset X$ and any $\eps>0$ there is a $\delta>0$ so that for all $x,y \in S$ satisfying $\|x-y\|_X\leq \delta$, for any $t\in I$ and any $d\in \Dc$ it holds that 
%\begin{eqnarray}
%\|\phi(t,x,d)-\phi(t,y,d)\|_X\leq \varepsilon.
%\label{eq:Uniform_continuity}
%\end{eqnarray}
%\end{itemize}
Then for any $\rho\in\K\backslash\Kinf$ so that $\rho(r)\leq \alpha(r)$ for all $r\in\R_+$ it holds that 
\begin{eqnarray}
V(x):= \sup_{d\in\Dc} \int_0^\infty \rho(\|\phi(s,x,d)\|_X)ds
\label{eq:ncLF_discontinuous_construction}
\end{eqnarray}
%\begin{eqnarray}
%V(x):= \sum_{k=1}^\infty 2^{-k} \sup_{d\in\Dc} \int_0^\infty G_k \circ \rho(\|\phi(s,x,d)\|_X)ds
%\label{eq:ncLF_construction}
%\end{eqnarray}
is a (possibly non-coercive) Lyapunov function for $\Sigma$, satisfying \eqref{eq:1} with $\psi_2$ as above and so that
\eqref{eq:V_estimate_along_trajectory_final} holds.
\end{theorem}

Before we proceed to the proof of
Theorem~\ref{thm:converse_Non-coercive_Lyapunov_theorem}, we would like to
stress, that in contrast to most of the converse Lyapunov theorems for
infinite-dimensional nonlinear systems (as \cite{Mir17a},
\cite[Section 3.4]{KaJ11b}), we do not impose any additional regularity
assumptions on the flow of the system, in particular, we assume neither
continuous dependence on data, nor robustness of the equilibrium point,
nor the RFC property. 
% This makes Theorem~\ref{thm:converse_Non-coercive_Lyapunov_theorem} a
% pure non-coercive converse Lyapunov theorem, since the assumptions made
% in the theorem are far weaker than needed to ensure an existence of a
% coercive (local or global) Lyapunov function.
Theorems~\ref{thm:direct_Non-coercive_Lyapunov_theorem} and~\ref{thm:converse_Non-coercive_Lyapunov_theorem} together show that noncoercive Lyapunov functions are a natural tool for analysis of integral stability properties. 

%For brevity, the proof of Theorem~\ref{thm:converse_Non-coercive_Lyapunov_theorem} is
%only sketched. 
In the proof we follow ideas from \cite[Section 3.4]{KaJ11b},
\cite[Theorem 5.6]{Mir17a}.

\begin{proof}(of Theorem~\ref{thm:converse_Non-coercive_Lyapunov_theorem}).

%Assume that properties (i) and (ii) in the formulation of the theorem hold. 
Pick any $\rho\in\K\backslash\Kinf$ so that $\rho(r)\leq \alpha(r)$ for all $r\in\R_+$.
 %and define for any $k\in\N$
%\begin{eqnarray}
%V_k(x):= \sup_{d\in\Dc} \int_0^\infty G_k \circ \rho(\|\phi(s,x,d)\|_X)ds.
%\label{eq:V_k_construction}
%\end{eqnarray}
\textbf{(i).} Since $\Sigmafp$ is iUGS with $\alpha,\psi$, it follows that
\begin{eqnarray*}
0\leq V(x) \leq \sup_{d\in\Dc} \int_0^\infty \alpha(\|\phi(s,x,d)\|_X)ds\leq \psi_2(\|x\|_X).
\end{eqnarray*}
%The properties $V(0)=0$, $V(x) > 0$ for $x \neq 0$ are relatively easy to obtain.

Since $0$ is an equilibrium of $\Sigma$, $\phi(s,0,d)\equiv 0$ for all
$s\geq 0$ and all $d\in\Dc$, which immediately implies that $V(0)=0$. If
$x \neq0$, then by continuity of solutions, we have for every $d
\in D$ a $T>0$ such that $\phi(t,x,d) \neq0$ for all $t\in [0,T]$. This
implies that $V(x) > 0$ and in summary \eqref{eq:1} is satisfied.

% Assume that $V(x) = 0$ for a certain $x \in X$. Then for every $d\in \Dc$
% we have $\int_0^\infty \alpha(\|\phi(s,x,d)\|_X)ds = 0$. Assume that
% $\phi(t,x,d) \neq 0$ for some $t>0$. Then due to continuity of $t\mapsto
% \phi(t,x,d)$ there is $\eps>0$ so that $\phi(s,x,d)\neq 0$ for all $s\in
% (t-\eps,t+\eps)$, and thus $\int_0^\infty \alpha(\|\phi(s,x,d)\|_X)ds >
% 0$, a contradiction. Hence $\phi(\cdot,x,d) \equiv 0$, and the identity
% axiom shows that $x= \phi(0,x,d)=0$. Overall, $V(x)>0$ whenever $x\neq 0$.

%Next we are going 
\textbf{(ii).} To compute the Dini derivative of $V$, fix $x \in X$ and $v \in \Dc$. In view of the cocycle property we have for any $h>0$:
\begin{eqnarray*}
V\big(\phi(h,x,v)\big) &=& \sup_{d \in \Dc} \int_0^{\infty} \rho(\|\phi(t,\phi(h,x,v),d)\|_X) dt\\
&=& \sup_{d \in \Dc} \int_0^{\infty} \rho(\|\phi(t+h,x,\tilde{d})\|_X) dt \\
&=& \sup_{d \in \Dc} \int_h^{\infty} \rho(\|\phi(t,x,\tilde{d})\|_X) dt,
\end{eqnarray*}
where the disturbance function $\tilde{d}$ is defined as 
\[
\tilde{d}(t) := 
\begin{cases}
v(t), & \text{ if } t \in [0,h] \\ 
d(t-h)  & \text{ otherwise}.
\end{cases}
\]
Note that $\tilde{d}\in \Dc$ due to the axiom of concatenation.
Since $\tilde{d}(t) = v(t)$ for $t\in[0,h]$, it holds that
\begin{align*}
\int_0^h &\rho(\|\phi(t,x,v)\|_X) dt + V\big(\phi(h,x,v)\big) \\
&= \sup_{d \in \Dc} \Big(\int_0^h \rho(\|\phi(t,x,v)\|_X) dt +  \int_h^{\infty} \rho(\|\phi(t,x,\tilde{d})\|_X) dt\Big)\\
&= \sup_{d \in \Dc} \Big(\int_0^h \rho(\|\phi(t,x,\tilde{d})\|_X) dt +  \int_h^{\infty} \rho(\|\phi(t,x,\tilde{d})\|_X) dt\Big)\\
&= \sup_{d \in \Dc} \int_0^{\infty} \rho(\|\phi(t,x,\tilde{d})\|_X) dt.
\end{align*}
%\begin{multline*}
%\int_0^h \rho(\|\phi(t,x,v)\|_X) dt + V\big(\phi(h,x,v)\big) %\\ 
%&=& \sup_{d \in \Dc} \Big(\int_0^h \rho(\|\phi(t,x,v)\|_X) dt +  \int_h^{\infty} \rho(\|\phi(t,x,\tilde{d})\|_X) dt\Big)\\
%&=& \sup_{d \in \Dc} \Big(\int_0^h \rho(\|\phi(t,x,\tilde{d})\|_X) dt +  \int_h^{\infty} \rho(\|\phi(t,x,\tilde{d})\|_X) dt\Big)\\
%&=& \sup_{d \in \Dc} \int_0^{\infty} \rho(\|\phi(t,x,\tilde{d})\|_X) dt.
%\end{multline*}
%\begin{multline*}
%\int_0^h \rho(\|\phi(t,x,v)\|_X) dt + V\big(\phi(h,x,v)\big) \\ 
%%&=& \sup_{d \in \Dc} \Big(\int_0^h \rho(\|\phi(t,x,v)\|_X) dt +  \int_h^{\infty} \rho(\|\phi(t,x,\tilde{d})\|_X) dt\Big)\\
%%&=& \sup_{d \in \Dc} \Big(\int_0^h \rho(\|\phi(t,x,\tilde{d})\|_X) dt +  \int_h^{\infty} \rho(\|\phi(t,x,\tilde{d})\|_X) dt\Big)\\
%= \sup_{d \in \Dc} \int_0^{\infty} \rho(\|\phi(t,x,\tilde{d})\|_X) dt.
%\end{multline*}
Since the supremum cannot decrease, if we allow a larger class of
disturbances, it may be seen that for all $h>0$ we have %holds that
\begin{align}
\label{eq:Bellman-principle}
\int_0^h \rho(\|\phi(t,x&,v)\|_X) dt + V\big(\phi(h,x,v)\big) \nonumber\\
&\leq \sup_{d \in \Dc} \int_0^{\infty} \rho(\|\phi(t,x,d)\|_X) dt
= V(x).
\end{align}
The obtained inequality may be interpreted as an instance of Bellman's principle.
%\textbf{(i).}
% First of all, it helps 
To compute the Dini derivative of $V$ along trajectories we note that
rearranging the inequality \eqref{eq:Bellman-principle} we obtain for all
$h>0$ that
\begin{equation}
\frac{1}{h}\big( V\big(\phi(h,x,v)\big) - V(x) \big) \leq -\frac{1}{h}\int_0^h \rho(\|\phi(t,x,v)\|_X) dt.
\label{eq:Auxiliary_Estimate_LF}
\end{equation}
As the map $t\mapsto \rho(\|\phi(t,x,v)\|_X)$ is continuous by the axiom
of continuity, it follows that 
%  Let us compute the limit of the integral in the right hand side of \eqref{eq:Auxiliary_Estimate_LF}. The identity axiom tells that $\phi(0,x,d_1)=x$, which implies that
%  \begin{align*}
%   \Big| \rho(\|x\|_X)&  - \frac{1}{h}\int_0^{h} \rho(\|\phi(t,x,d_1)\|_X) dt \Big|\\
%  &=
%   \Big| \frac{1}{h}\int_0^{h}  \rho(\|x\|_X) - \rho(\|\phi(t,x,d_1)\|_X) dt \Big|\\
%  &\leq
%   \frac{1}{h}\int_0^{h} \big|   \rho(\|\phi(0,x,d_1)\|_X) - \rho(\|\phi(t,x,d_1)\|_X) \big| dt.
%  \end{align*}
%  Furthermore, according to the continuity axiom, the map $t\mapsto \phi(t,x,d_1)$ is continuous. Since $\rho$ and $x \mapsto \|x\|_X$ are continuous functions, the function $t \mapsto \rho(\|\phi(t,x,d_1)\|_X)$ is again continuous.

%  In particular, for any $\eps>0$ there is $h>0$ so that for any $t\in[0,h]$ we have
%  \[
%  \big|\rho(\|\phi(0,x,d_1)\|_X) - \rho(\|\phi(t,x,d_1)\|_X)\big| \leq \eps.
%  \]
%  For this $h>0$ we have in view of the above facts that 
%  \[
%  \Big| \rho(\|x\|_X) - \frac{1}{h}\int_0^{h} \rho(\|\phi(t,x,d_1)\|_X) dt \Big| \leq \eps.
%  \]
%  Since $\eps>0$ can be chosen arbitrarily small, this means that
 \begin{equation*}
 \mathop{\lim} \limits_{h \rightarrow +0} \frac{1}{h}\int_0^{h} \rho(\|\phi(t,x,v)\|_X) dt
 =  \rho(\|x\|_X),
 \end{equation*}
 and due to \eqref{eq:Auxiliary_Estimate_LF} we obtain
%which may be used (with some care) to obtain
%\begin{eqnarray*}
$\dot{V}_v(x) \leq -  \rho(\|x\|_X)$ and so
\eqref{DissipationIneq_UGAS_With_Disturbances} holds for the given $\rho \in {\cal K}$.
%\end{eqnarray*}
% Also property \eqref{eq:V_estimate_along_trajectory_final} may be obtained
% using \eqref{eq:Auxiliary_Estimate_LF}.
%

 \textbf{(iii).} It remains to show that
 \eqref{eq:V_estimate_along_trajectory_final} holds. The right hand side
 in this inequality holds (even in $s=0$) by a direct application of
 \eqref{eq:Bellman-principle} to the case $x=\phi(s,y,d)$ and arbitrary
 $s\geq0,  y\in X, d\in{\cal D}$, because the integral on the left hand
 side of \eqref{eq:Bellman-principle} is always nonnegative.

 % Pick any $d\in\Dc$, any $s\geq 0$ and any $y\in X$. Now substitute $x:=\phi(s,y,d)$ and $v:=d(s+\cdot)$ into the Bellman's inequality. We obtain
%  \begin{align}
%  \label{eq:Bellman-principle_2}
%  \int_0^h \rho(\|&\phi(t,\phi(s,y,d),d(s+\cdot))\|_X) dt \nonumber\\
% & + V\big(\phi(h,\phi(s,y,d),d(s+\cdot))\big) \leq V(\phi(s,y,d)).
%  \end{align}
%  Using the cocycle property and rearranging the terms of the above inequality, we conclude
%  \begin{eqnarray*}
%  V\big(\phi(s+h,y,d)\big) - V(\phi(s,y,d)) \leq -\int_0^h \rho(\|\phi(t+s,y,d)\|_X) dt.
%  \end{eqnarray*}
%  Since $t\mapsto \phi(s+t,y,d)$ is a continuous map (due to continuity axiom of $\Sigma$),
%  $\int_0^h \rho(\|\phi(t,y,v)\|_X) dt \to 0$ as $h\to +0$. Hence, for any $s\geq 0$, any $y\in X$ and any $d\in\Dc$ we have
%  \begin{eqnarray}
%  \label{eq:V_estimate_along_trajectory_1}
%  \mathop{\underline{\lim}} \limits_{h \rightarrow +0} V\big(\phi(s+h,y,d)\big) \leq V(\phi(s,y,d)).
%  \end{eqnarray}

To show the left hand side, fix  $d\in\Dc$,  $s>0$,  $y\in X$ and $h\in (0,s)$.  Substitute $x:=\phi(s-h,y,d)$ and $v:=d(s-h+\cdot)$ into Bellman's inequality \eqref{eq:Bellman-principle}
to obtain
 \begin{align*}
 %\label{eq:Bellman-principle_3}
 \int_0^h \rho\big(\|\phi(t,\phi(s-h,y,d),d(s-h+\cdot))\|_X\big) dt 
&+ V\big(\phi(h,\phi(s-h,y,d),d(s-h+\cdot))\big) \\
&\leq V(\phi(s-h,y,d)).
 \end{align*}
 Using again the cocycle property and rearranging the terms of the above
 inequality, we conclude for all $h\in (0,s)$ that
 \begin{eqnarray*}
  V(\phi(s-h,y,d)) \geq  V\big(\phi(s,y,d)\big) + \int_0^h \rho(\|\phi(t+s-h,y,d)\|_X) dt.
 \end{eqnarray*}
 Arguing as above, we obtain that:
 \begin{eqnarray}
 \label{eq:V_estimate_along_trajectory_2}
 \mathop{\underline{\lim}} \limits_{h \rightarrow +0} V\big(\phi(s-h,y,d)\big) \geq V\big(\phi(s,y,d)\big).
 \end{eqnarray}

 This shows \eqref{eq:V_estimate_along_trajectory_final} and the proof is complete.
\end{proof}

\section{Conclusions}

In order to understand the implications of the existence of non-coercive
Lyapunov functions we have introduced several integral notions of
stability, which do not measure the pointwise distance to the equilibrium but
rather a weighted average along trajectories. It has been shown that in a
quite general setting noncoercive Lyapunov functions characterize these
integral notions. Also the relation to standard stability notions are
discussed, see also Figure~\ref{fig:Stability}. 
It will be of interest to investigate how the results obtained
here carry over to questions of input-to-state stability (ISS). 
Some results in this direction have been recently developed in \cite{MiW18b,JMP18}.

We point out that in \cite{Mir17a} the relation between uniform weak attractivity and closely related concepts of weak
attractivity and recurrence are discussed for systems without inputs.

 \begin{figure*}[tbh]
 \centering
 \begin{tikzpicture}[>=implies,thick]

 %\node (Unif_Level) at (-4,5.5) {Uniform level};
 %\node (Nonunif_Level) at (-4.25, 	2.5) {Non-uniform level};

 \node (UGAS) at (3,8) {UGAS};
 \node (iUGATT_REP_RFC) at (7,8) {iUGATT $\wedge$ 0-REP $\wedge$ RFC};
 \node (UGATT_REP_RFC) at (12.5,8) {UGATT $\wedge$ 0-REP $\wedge$ RFC};

 \draw [rounded corners] (2.5,8.8) rectangle (14.9,7.6);
 \node (Thm_UGAS_Charact) at (8.6,8.5) {\footnotesize Theorem~\ref{thm:UGAS_characterization}};

 \node (ncLF) at (0.4,6) {$\exists$ a noncoercive LF};
 \node (iUGAS) at (3.4,4) {iUGAS};
 %\node (iUGATT_iULS) at (-1,4) {iUGATT $\wedge$ iULS};
 \node (iUGATT_iREP) at (6.3,4) {iUGATT $\wedge$ iREP};
 \node (UGWA_iULS) at (-1,4) {UGWA $\wedge$ iULS};
 \node (iUGS) at (1.6,4) {iUGS};
 \node (iULS) at (-2,1.8) {iULS};
 \node (iUGATT) at (7,1.8) {iUGATT};
 \node (UGWA_Ult_iULS) at (4,1.8) {UGWA $\wedge$ Ult-iULS};
 \node (iREP_Ult_iULS) at (0.5,1.8) {iREP $\wedge$ Ult-iULS};

 \draw [rounded corners] (5.2,6.8) rectangle (12.2,5.6);

 \node (iUGATT_REP) at (7,6) {iUGATT $\wedge$ 0-REP};
 \node (iUGATT_ULS) at (10.5,6) {iUGATT $\wedge$ ULS};
 \node (UGATT_REP) at (14,4) {UGATT $\wedge$ 0-REP};
 \node (UGATT_ULS) at (10.5,4) {UGATT $\wedge$ ULS};
 \node (UGATT) at (10.5,1.8) {UGATT};
 \node (UGWA) at (7,0.5) {UGWA};

 \draw [rounded corners] (-2.5,4.8) rectangle (7.8,3.5);
 \node (Thm5) at (3,4.5) {\footnotesize Theorem~\ref{thm:iUGAS_characterization}};

 %\node (Thm3) at (5.5,5.25) {\footnotesize Theorem~\ref{thm:direct_Non-coercive_Lyapunov_theorem}};
 \node (Lemma_REP_iREP) at (5.5,5.25) {\footnotesize Lemma~\ref{lem:REP_implies_iREP} };

 \node (Prop_iUGATT_REP) at (8.6,6.5) {\footnotesize Proposition~\ref{prop:NewProp_and_REP_imply_UGATT}};
 \node (Prop_iUGATT_REP_2) at (11.8,5.2) {\footnotesize Proposition~\ref{prop:NewProp_and_REP_imply_UGATT}};

 \node (Prop_NC_LF) at (-1.25,5.35) {\footnotesize Theorem~\ref{thm:direct_Non-coercive_Lyapunov_theorem}}; 
 \node (Prop_NC_LF_Conv) at (2.2,5.35) {\footnotesize Theorem~\ref{thm:converse_Non-coercive_Lyapunov_theorem}};

 \node (REP) at (8.6,0.5) {REP};
 %\draw[->,double equal sign distance] (UGAS) to (6.5,4.9);
 \draw[->,double equal sign distance] (7,5.5) to (7,4.9);
 \draw[->,double equal sign distance] (8.6,5.5) to (REP);

 \draw[->,double equal sign distance] (9.5,3.4) to (REP);

 %\draw[->,double equal sign distance] (iUGATT_REP) to (7,4.9);
 \draw[<->,double equal sign distance] (UGAS) to (iUGATT_REP_RFC);
 \draw[<->,double equal sign distance] (iUGATT_REP_RFC) to (UGATT_REP_RFC);

%Non-coercivity results
%Direct theorem
 \draw[->,double equal sign distance] (-0.2,5.7) to (-0.2,4.9);
%Converse theorem
 \draw[<-,double equal sign distance] (1.2,5.7) to (1.2,4.9);

 \draw[<->,double equal sign distance] (UGWA_iULS) to (iUGS);
 \draw[<->,double equal sign distance] (iUGS) to (iUGAS) ;
 \draw[<->,double equal sign distance] (iUGAS)  to (iUGATT_iREP);

 %\draw[->,double equal sign distance] (iUGATT_REP_RFC) to (iUGATT_REP);
 \draw[->,double equal sign distance] (8.6,7.5) to (8.6,6.9);
 %\draw[->,double equal sign distance] (iUGATT_REP) to (iUGATT);
 \draw[<->,double equal sign distance] (iUGATT_REP) to (iUGATT_ULS);

 \draw [rounded corners] (9.1,4.8) rectangle (15.6,3.5);

 %\draw[->,double equal sign distance] (iUGATT_ULS) to (UGATT_ULS);
 %\draw[->,double equal sign distance] (UGATT_ULS) to (UGATT);
 \draw[->,double equal sign distance] (10.4,5.5) to (10.4,4.9);
 \draw[->,double equal sign distance] (12.7,3.4) to (12.7,2.8);
 \draw[<->,double equal sign distance] (UGATT_ULS) to (UGATT_REP);

 %\draw[->,double equal sign distance](7,3.4) to (iUGATT);
 \draw[<->,double equal sign distance] (iUGATT) to (UGWA_Ult_iULS);
 %\draw[->,double equal sign distance] (iUGATT) to (UGWA);
 \draw[->,double equal sign distance] (7,1.3) to (7,0.7);
 \draw [rounded corners] (2.35,1.4) rectangle (7.8, 2.6);

 \node (UGWA_UltULS) at (14,1.8) {UGWA $\wedge$ UltULS};
 \node (UltULS) at (12.7,0.5) {UltULS};

 \draw[->,double equal sign distance] (5.1,3.4) to (5.1,2.7);
 \draw[->,double equal sign distance] (UGATT) to (UGWA);
 \draw[<->,double equal sign distance] (UGATT) to (UGWA_UltULS);
 \node (Prop_iUGATT_Crit) at (5.1,2.25) {\footnotesize Proposition~\ref{prop:iUGATT_equiv_UGWA_Ult_iULS}};
 \draw [rounded corners] (9.7,1.4) rectangle (15.6, 2.6);

 \node (Prop_UGATT_Crit) at (12.7,2.25) {\footnotesize Proposition~\ref{prop:UGATT_equiv_UGWA_Ult_ULS}};

 \draw[->,double equal sign distance] (12.7,1.3) to (UltULS);

 \node (iREP) at (-0.25,0.5) {iREP};
 \node (Ult_iULS) at (2.2,0.5) {Ult-iULS};

 \draw[->,double equal sign distance] (-0.25,3.4) to (-0.25,2.7);
 %\draw[->,double equal sign distance] (UGWA_Ult_iULS) to (Ult_iULS);
 \draw[->,double equal sign distance] (4,1.3) to (Ult_iULS);

 \node (Prop_iUGATT_Crit) at (-0.2,2.25) {\footnotesize Lemma~\ref{lem:iULS_criterion} };

 \draw[<->,double equal sign distance] (iULS) to (iREP_Ult_iULS);
 %\draw[->,double equal sign distance] (iULS) to (Ult_iULS);
 \draw[->,double equal sign distance] (0,1.3) to (Ult_iULS);
 \draw[->,double equal sign distance] (-0.25,1.3) to (iREP);
 \draw [rounded corners] (-2.5,1.4) rectangle (1.9, 2.6);

 \end{tikzpicture}

 \caption{Relations between stability notions}%
 \label{fig:Stability}%
 \end{figure*}
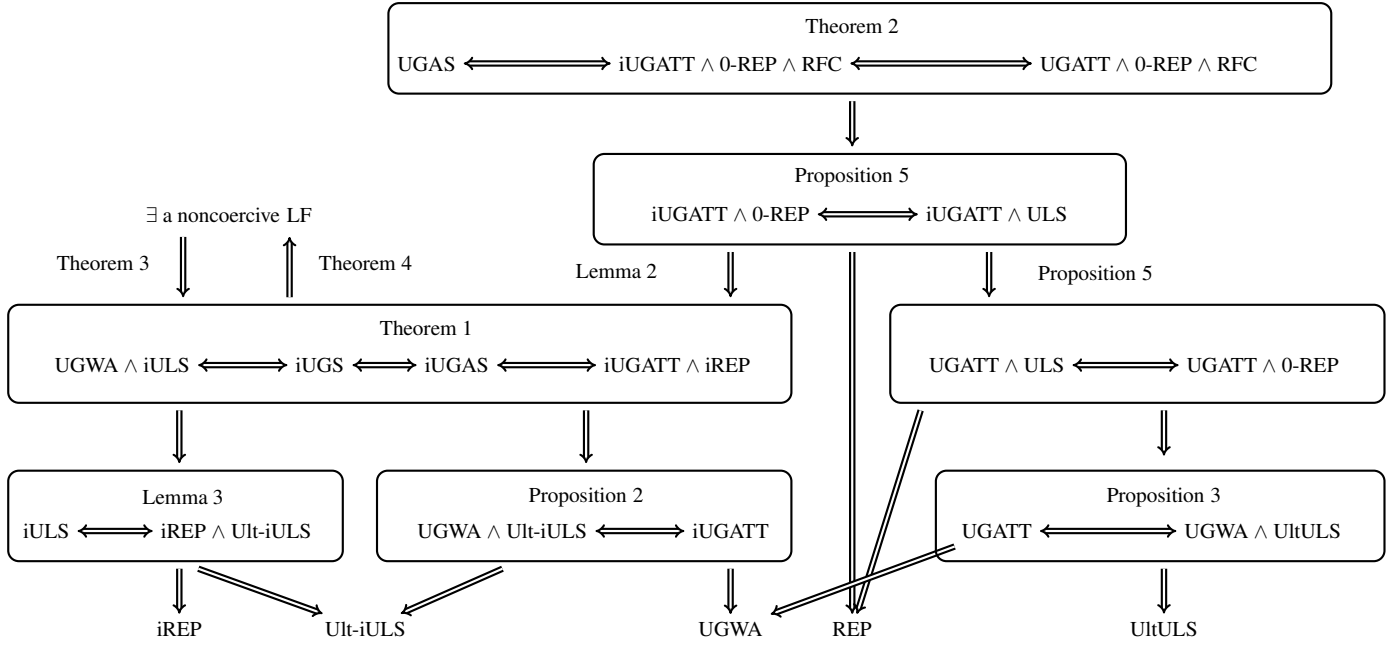

\appendix
\renewcommand\thesection{\Alph{section}}
  \section{Appendix}
\label{sec:appendix}
\renewcommand{\thetheorem}{\thesection.\arabic{theorem}}%

  In this appendix, we show a result providing conditions for the
  existence of a $\KL$-bound for a function of two arguments that has been
  used in our proofs. Although the result may not be surprising for the
  expert, we have not found an explicit reference and so prefer to present
  the construction here.

  We start with an auxiliary statement: 
  \begin{proposition}
  \label{prop:Partition_of_Reals}
  For any $z\in C(\R_+,\R)$ and for any $\eps>0$ there is a 
%partition 
sequence $\{R_k\}_{k\in\Z} \subset (0,+\infty)$ satisfying the following properties:
  \begin{itemize}
      \item[(i)] $R_k \to 0$ as $k\to -\infty$.
      \item[(ii)] $R_k\to+\infty$ as $k\to+\infty$.
      \item[(iii)] $R_k<R_{k+1}$ for all $k\in\Z$.
      %\item[(iv)] zero is the only finite accumulation point of $\{R_k\}_{k\in\Z}$.
      \item[(iv)] $\max_{s\in[R_k,R_{k+1}]}z(s) - \min_{s\in[R_k,R_{k+1}]}z(s) <\eps$, for all $k\in\Z$.
  \end{itemize}
  \end{proposition}

We stress that by conditions (i), (ii), (iii), zero is the only finite accumulation point of $\{R_k\}_{k\in\Z}$.

  \begin{proof}
  It is easy to see that a sequence satisfying the properties (i)--(iii) always exists 
  (and can be chosen independently on $z$). Pick any such sequence and denote it $\{S_k\}_{k\in\Z}$. 
  Since $z\in C(\R_+,\R)$, by Cantor's theorem $z$ is uniformly continuous on $[S_k,S_{k+1}]$ for any $k\in\Z$.
  Thus, there exists a partition of $[S_k,S_{k+1}]$ into finitely many
  subintervals with boundary points $S_k=S_{k1} < S_{k2} \ldots < S_{km(k)}=S_{k+1}$ so that $\max_{s\in[S_{ki},S_{ki+1}]}z(s) - \min_{s\in[S_{ki},S_{ki+1}]}z(s) <\eps$, for each $i=1,\ldots,m(k)-1$.

  Now define the desired sequence $\{R_k\}_{k\in\Z}$ by inserting 
  considering the ordered sequence $\{S_{kj}\}_{k\in\Z,
    j=1,\ldots,mk}$. Clearly, $\{R_k\}_{k\in\Z}$ satisfies (i)-(iv).
  \end{proof}

  The following estimation result is useful in our derivations.
  \begin{proposition}
  \label{prop:Upper_estimate_as_KL_function}
  Let $\psi:\R_+ \times \R_+ \to \R_+$ be any function which is
  nondecreasing and continuous at 0 in the first argument, nonincreasing
  in the second argument and so that $\lim_{t\to\infty}\psi(r,t)= 0$ for
  any $r\geq 0$. Let also $\psi(0,t)=0$ for any $t\geq 0$. Then there
  exists a $\beta\in\KL$ such that 
  \[
  \psi(r,t)\leq \beta(r,t) \quad \forall r,t\geq \R_+.
  \]
  \end{proposition}

  \begin{proof}
  Pick a sequence $R:=\{R_k\}_{k\in\Z} \subset (0,+\infty)$ satisfying the properties (i)-(iii) of Proposition~\ref{prop:Partition_of_Reals} and
  another sequence $\tau:=\{\tau_m\}_{m\in\N} \subset [0,+\infty)$
  satisfying the properties (ii)-(iii) of
  Proposition~\ref{prop:Partition_of_Reals} and so that $\tau_0=0$. The Cartesian product $R\times\tau$ defines a mesh over $\R_+\times\R_+$.
  Let $\omega\in\KL$ be arbitrary so that $\omega(r,t) > 0$ for all $(r,t)
  \in (0,\infty) \times \R_+$.

  For each $k\in\Z$ and $m\in\N$, $m\neq 0$ define
  \[
  \beta(R_k,\tau_m):=\psi(R_{k+1},\tau_{m-1}) + \omega(R_{k+1},\tau_{m-1}).
  \]
  For $m=0$ define
  \[
  \beta(R_k,0)=\beta(R_k,\tau_0):=2\psi(R_{k+1},0) + \omega(R_{k+1},0).
  \]
  For each $k\in\Z$ and $m\in\N$ define $\beta(r,t)$ for all $(r,t)$ in
  the triangles with corner points ($(R_{k+1},\tau_{m})$, $(R_{k},\tau_{m+1})$, $(R_{k},\tau_{m})$) or
  ($(R_{k+1},\tau_{m})$, $(R_{k+1},\tau_{m+1})$, $(R_{k},\tau_{m+1})$)
  by linear interpolation of the values in the corner points (which have
  already been defined).
  % a plane which intersects the values of $\beta$ at the vertices (such
  % plane is unique).
  This defines the values of $\beta$ in $(0,+\infty)\times[0,+\infty)=\bigcup_{k\in\Z}\bigcup_{m\in\N}[R_k,R_{k+1}]\times[\tau_k,\tau_{k+1}]$.

  Defining $\beta(0,t):=0, t\geq 0,$ we see that $\beta$ is defined over $\R_+\times\R_+$, is continuous, strictly increasing in the first argument and decreasing in the second argument. Since $\lim_{t\to\infty}\psi(r,t)= 0$ for any $r\geq 0$, it holds also that 
  $\lim_{t\to\infty}\beta(r,t)= 0$ for any $r\geq 0$. Overall, $\beta\in\KL$.

  It remains to show that $\beta$ estimates $\psi$ from above.
  To see this, pick any $k\in\Z$ and $m\in\N$. For every $(r,s)\in [R_k,R_{k+1}]\times[\tau_m,\tau_{m+1}]$ we have:
  \begin{eqnarray*}
  \beta(r,s) - \psi(r,s) &\geq& \beta(R_k,\tau_{m+1}) - \psi(R_{k+1},\tau_{m}) \\
  &=& \psi(R_{k+1},\tau_{m}) + \omega(R_{k+1},\tau_{m}) - \psi(R_{k+1},\tau_{m})\\
  &=& \omega(R_{k+1},\tau_{m}).
  \end{eqnarray*}
  Hence, $\beta(r,s) \geq \psi(r,s)$ for all $r,s\geq0$.  
  \end{proof}

\begin{acknowledgements}
This research has been supported by the German Research Foundation (DFG) within the project 
``Input-to-state stability and stabilization of distributed parameter systems'' (grant Wi 1458/13-1).
\end{acknowledgements}

%\bibliographystyle{abbrv}
%\bibliography{Mir_LitList}

\end{document}